\documentclass[10pt,reqno]{amsart}
\usepackage{hyperref}
\usepackage{latexsym,amsmath}
\usepackage{enumerate}
\usepackage{amsfonts}
\usepackage{amssymb}
\usepackage{latexsym}
\usepackage{fixmath}
 \usepackage[usenames,dvipsnames]{color}

\usepackage{graphicx}

\usepackage[T1]{fontenc}
\usepackage{fourier}
\usepackage{bbm}
\usepackage{color}
\usepackage{fullpage}
\usepackage{todonotes}
\numberwithin{equation}{section}

\renewcommand{\vec}[1]{\boldsymbol{#1}}
\renewcommand{\Im}{\mathrm{im}}
\renewcommand{\subset}{\subseteq}

\newcommand\vy{\vec y}

\newcommand\vx{\vec x}

\newcommand\va{\vec a}
\newcommand\vi{\vec i}

\newcommand\vI{\vec I}
\newcommand\vhx{\vec{\hat x}}
\newcommand\vhy{\vec{\hat y}}
\newcommand\vhA{\vA}

\newcommand\CPC{Combinatorics, Probability and Computing}

\newcommand\vm{\vec m}
\newcommand\NU{\vec\nu}

\newcommand\GAMMA{{\vec\gamma}}

\newcommand\nix{\,\cdot\,}

\newcommand\vA{\vec A}
\newcommand\vB{\vec B}

\newcommand\dd{{\mathrm d}}
\newcommand\G{\vec G}

\newcommand\SIGMA{\vec\sigma}
\newcommand\TAU{\vec\tau}
\newcommand\aco[1]{#1}

\newcommand\cA{\mathcal{A}}
\newcommand\cB{\mathcal{B}}
\newcommand\cC{\mathcal{C}}

\newcommand\cF{\mathcal{F}}

\newcommand\cE{\mathcal{E}}

\newcommand\cQ{\mathcal{Q}}

\newcommand\cS{\mathcal{S}}

\newcommand\cJ{\mathcal{J}}

\newcommand\cM{\mathcal{M}}
\newcommand\cO{\mathcal{O}}
\newcommand\cP{\mathcal{P}}

\newcommand\cY{\mathcal{Y}}

\newcommand\cZ{\mathcal{Z}}

\def\cC{{\mathcal C}}
\def\cE{{\mathcal E}}

\newcommand\eul{\mathrm{e}}
\newcommand\eps{\varepsilon}

\newcommand\ZZ{\mathbb{Z}}
\newcommand\FF{\mathbb{F}}

\newcommand\Var{\mathrm{Var}}
\newcommand\Erw{\mathbb{E}}
\newcommand{\vecone}{\vec{1}}

\newcommand{\Po}{{\rm Po}}

\newcommand\TV[1]{\left\|{#1}\right\|_{\mathrm{TV}}}
\newcommand\tv[1]{\|{#1}\|_{\mathrm{TV}}}

\newcommand{\bink}[2] {{\binom{#1}{#2}}}
\newcommand\bc[1]{\left({#1}\right)}
\newcommand\cbc[1]{\left\{{#1}\right\}}
\newcommand\bcfr[2]{\bc{\frac{#1}{#2}}}
\newcommand{\bck}[1]{\left\langle{#1}\right\rangle}
\newcommand\brk[1]{\left\lbrack{#1}\right\rbrack}
\newcommand\scal[2]{\bck{{#1},{#2}}}

\newcommand\abs[1]{\left|{#1}\right|}

\newcommand{\whp}{w.h.p.}

\newcommand{\stacksign}[2]{{\stackrel{\mbox{\scriptsize #1}}{#2}}}
\newcommand{\tensor}{\otimes}

\newcommand{\Erdos}{Erd\H{o}s}
\newcommand{\Renyi}{R\'enyi}

\newcommand\pr{\mathbb{P}} 
 
\newcommand\Lem{Lemma}
\newcommand\Prop{Proposition}
\newcommand\Thm{Theorem}

\newcommand\Cor{Corollary}
\newcommand\Sec{Section}
\newcommand\Chap{Chapter}

\newcommand\fZ{\mathfrak{Z}}

\newtheorem{definition}{Definition}[section]
\newtheorem{claim}[definition]{Claim}

\newtheorem{theorem}[definition]{Theorem}
\newtheorem{lemma}[definition]{Lemma}
\newtheorem{proposition}[definition]{Proposition}
\newtheorem{corollary}[definition]{Corollary}

\DeclareMathOperator{\nul}{nul}
\DeclareMathOperator{\rk}{rk}

\begin{document}

\title{The satisfiability threshold for random linear equations}

\author{Peter Ayre, Amin Coja-Oghlan, Pu Gao, No\"ela M\"uller}
\thanks{Gao's research is supported by ARC DE170100716 and ARC DP160100835.}

\address{Peter Ayre, {\tt peter.ayre@unsw.edu.au}, School of Mathematics and Statistics, UNSW Sydney, NSW 2052, Australia}

\address{Amin Coja-Oghlan, {\tt acoghlan@math.uni-frankfurt.de}, Goethe University, Mathematics Institute, 10 Robert Mayer St, Frankfurt 60325, Germany.}

\address{Pu Gao, {\tt jane.gao@monash.edu},   School of Mathematical Sciences,    Monash University,  Australia}

\address{No\"ela M\"uller, {\tt nmueller@math.uni-frankfurt.de}, Goethe University, Mathematics Institute, 10 Robert Mayer St, Frankfurt 60325, Germany.}

\begin{abstract}
Let $\vA$ be a random $m\times n$ matrix over the finite field $\FF_q$ with precisely $k$ non-zero entries per row and let $\vy\in\FF_q^m$ be a random vector chosen independently of $\vA$.
We identify the threshold $m/n$ up to which the linear system $\vA x=\vy$ has a solution with high probability and analyse the geometry of the set of solutions.
{In the special case $q=2$, known as the random $k$-XORSAT problem, the threshold was determined by
	[Dubois and Mandler 2002 for $k=3$, Pittel and Sorkin 2016 for $k>3$], and
	the proof technique was subsequently extended to the cases $q=3,4$ [Falke and Goerdt 2012].
But the argument depends on technically demanding second moment calculations that do not generalise to $q>4$.}
Here we approach the problem from the viewpoint of a decoding task, which leads to a transparent combinatorial proof.

\end{abstract}

\maketitle

\section{Introduction}\label{Sec_intro}

\subsection{Background and motivation}
Some of the longest-standing open problems in probabilistic combinatorics can be described along the following lines~\cite{ANP}.
There are a (large) number $x_1,\ldots,x_n$ of variables ranging over a finite domain $\Omega$.
They interact with each other through random constraints $a_1,\ldots,a_m$, drawn independently from the same probability distribution.
Each constraint binds a small number $k\geq2$ of variables and forbids them from taking certain value combinations.
The total number of constraints is a linear function of the number of variables, say, $m\sim dn/k$ for a fixed $d>0$, so that a variable is bound by $d$ constraints on the average.
In all but a few exceptional cases\footnote{These exceptional cases typically involve local structures whose emergence does not exhibit sharp threshold, whereas their appearance affects the satisfiability. For instance, a random $2$-XORSAT formula does not have a sharp satisfiability threshold. Whether it is satisfiable depends on the appearance of short cycles in the underlying graph of the formula. The probability of being satisfiable is away from both 0 and 1 when $d$ is strictly between 0 and 1.} it is conjectured that there exists a critical value $d_*>0$, the {\em satisfiability threshold}, such that
for $d<d_*$ there typically is an assignment of values to variables that satisfies all constraints simultaneously, while for $d>d_*$ no such
assignment is likely to exist.
In symbols,
	\begin{align*}
	\lim_{n\to\infty}\pr\brk{\exists\sigma\in\Omega^{\{x_1,\ldots,x_n\}}\,\forall i=1,\ldots,m:\sigma\models a_i}&=\begin{cases}
		1&\mbox{ if }d<d_*,\\
		0&\mbox{ if }d>d_*.
		\end{cases}
	\end{align*}
The key question is: what is the value of the satisfiability threshold $d_*$?

Even though the quest for satisfiability thresholds has received an enormous amount of attention (we will review the literature in \Sec~\ref{Sec_related}),
to this day the exact thresholds are known in only a handful of cases.
The {\em random $k$-XORSAT problem} was one of the first successes~\cite{DuboisMandler}.
Here the variables range over the field $\FF_2$ and there are random parity constraints, each involving $k$ randomly chosen variables.
More precisely, constraint $a_i$ requires that $\sum_{h=1}^kx_{\vec\alpha_{i,h}}=\vec y_i$, where $1\leq\vec\alpha_{i,1}<\cdots<\vec\alpha_{i,k}\leq n$ is a uniformly random sequence and $\vy_i\in\FF_2$ is chosen uniformly and independently of the $\vec\alpha_{i,h}$.
Thus, algebraically we ask whether a linear system $\vA x=\vy$ is likely to have a solution, where $\vA$ is a random $m\times n$ matrix  over the field $\FF_2$ with precisely $k$ non-zero entries per row and  $\vy\in\FF_2^m$ is chosen uniformly and independently of $\vA$.
Beyond the realm of probabilistic combinatorics, sparse random linear systems are of great importance in coding theory, computer science and statistical physics.
Indeed, `low-density parity check codes', a family of extraordinarily successful codes, are based on random linear systems~\cite{RichardsonUrbanke}.
A further application is the cuckoo hashing algorithm in computer science~\cite{Dietzfelbinger}.
In addition, random $k$-XORSAT is closely related to the diluted $k$-spin model of statistical mechanics, an intriguing spin glass model~\cite{MRTZ}.

The original derivation of the $k$-XORSAT threshold was based on a delicate second moment argument.
More precisely,  Dubois and Mandler~\cite{DuboisMandler} published a detailed second-moment-based proof for the case $k=3$ in 2002.
They claimed that their proof extends to $k>3$.
That, however, turned out to be far from straightforward.
A full proof for general $k$ was obtained by Pittel and Sorkin~\cite{PittelSorkin}.
Independently Dietzfelbinger et al.~\cite{Dietzfelbinger2} in a work on cuckoo hashing included an argument for general $k$
	(although the paper did not appear in a journal).
All of these papers lead to intricate analytical optimisation problems.
Moreover, the proofs from~\cite{Dietzfelbinger2,DuboisMandler,PittelSorkin} are specific to the field $\FF_2$.
A technically even far more demanding second-moment based extension to $\FF_3$ and $\FF_4$ was provided by Falke and Goerdt~\cite{GoerdtFalke}.
But their argument does not generalise to other fields either.

In this paper we establish the satisfiability threshold for random linear systems over any finite field $\FF_q$ by means of a transparent combinatorial argument,
thereby plugging what was arguably one of the most conspicuous gaps in the theory of random constraint satisfaction problems.
The key idea is to approach the problem from the angle of  coding theory. 
This point of view enables us to tackle not just the satisfiability threshold problem but also various related questions such as 
 the geometry of the set of solutions, for general finite fields without having to go through taxing calculations.
As an extension we also determine the satisfiability threshold for random systems of equations over finite Abelian groups.
Indeed, the proof strategy may extend to other random constraint satisfaction problems, a point that we revisit in \Sec~\ref{Sec_related}.

\subsection{Results} \label{Sec_Results}
Clearly, when we work with fields with more than two elements we need to specify how the non-zero entries of the random matrix are chosen.
We will allow for quite general distributions.
To be precise, let $q>1$ be a prime power, let $k\geq3$ be an integer and let $\FF_q^*=\FF_q\setminus\cbc0$ be the finite field of order $q$, excluding the zero element.
Further, let $P$ be a probability distribution on $\FF_q^{*\,k}$.
We assume that $P$ is permutation-invariant, i.e.,
	$P(\sigma_1,\ldots,\sigma_k)=P(\sigma_{\lambda(1)},\ldots,\sigma_{\lambda(k)})$
		for any $\sigma_1,\ldots,\sigma_k\in\FF_q^*$ and any permutation $\lambda$ of $\{1,\ldots,k\}$.
Moreover, for $d>0$ and an integer $n \geq k$ let $\vm=\vm(d,n)$ be a Poisson variable with mean $dn/k$.
Then we obtain a random $\vm\times n$ matrix $\vA$ over $\FF_q$ as follows.
Independently for each $i=1,\ldots,\vm$ choose a sequence $1\leq\vec\alpha_{i,1} < \cdots < \vec\alpha_{i,k}\leq n$ of integers uniformly at random
and independently choose a $k$-tuple $(\vec a_{i,1},\ldots,\vec a_{i,k})$ from the distribution $P$.
Then the $i$'th row of $\vA$ has the entries $\vec a_{i,h}$ in the positions $\vec\alpha_{i,h}$, while all other entries are zero.
Additionally, let $\vy\in\FF_q^{\vm}$ be a uniformly random vector, chosen independently of $\vA$ given $\vm$.

Let us highlight two interesting special cases.
If $P$ is the uniform distribution on $\FF_q^{*\,k}$, then $\vA$ is just a uniformly random matrix of size $\vm\times n$ with precisely $k$ non-zero entries per row.
Moreover, if $P$ gives probability one to the $k$-tuple $(1,\ldots,1)$, then $\vA$ is a random $\{0,1\}$-matrix with precisely $k$ non-zero entries per row.

For what $d,k$ is the random linear system $\vA x=\vy$ likely to have a solution?

\begin{theorem}\label{Thm_SAT}
Let $k\geq3$, let $q>1$ be a prime power and let $P$ be a permutation-invariant distribution on $\FF_q^{*\,k}$.
Set
	\begin{align}\label{eqrho}
	\rho_{k,d}&=\sup\cbc{x\in[0,1]:x=1-\exp(-dx^{k-1})}&\qquad\mbox{for $d>0$, and define}\\
		\label{eqdk}
	d_{k}&=\inf\cbc{d>0: 
		\rho_{k,d}-d\rho_{k,d}^{k-1}+(1-1/k)d\rho_{k,d}^k<0}.
	\end{align}
Then
	$$
	\lim_{n\to\infty}\pr\brk{\exists x\in\FF_q^n:\vA x=\vy}=\begin{cases}
		1&\mbox{ if }d<d_k,\\
		0&\mbox{ if }d>d_k\end{cases}
	$$
and thus for $d<d_k$ we have
	\begin{align}\label{eqrank}
	\lim_{n\to\infty}\pr\brk{\rk(\vA)=\vm}
	&=1.
	\end{align}
\end{theorem}
{\em Remark.} We will prove in Theorem~\ref{Thm_Mike} below that (a) $x=1-\exp(-dx^{k-1})$ has at most two positive fix-point solutions besides the trivial solution 0; (b) there is a critical value $d_k^{\star}$ where $\rho_{k,d}$ is strictly positive if $d\ge d_k^{\star}$ and is 0 if $d<d_k^{\star}$; (c) $\rho_{k,d}$ is a strictly increasing function of $d$ on $ [d_k^{\star},\infty]$.  Theorem~\ref{Thm_Mike} also implies that $\rho_{k,d}-d\rho_{k,d}^{k-1}+(1-1/k)d\rho_{k,d}^k<0$ has at most one root on $[d^{\star}_k,\infty)$.

\Thm~\ref{Thm_SAT} shows that the satisfiability threshold $d_k$ is independent of both $q$ and of the distribution $P$.
This is not surprising because the value $d_k$ has a natural combinatorial interpretation.
Namely, suppose that we subject the matrix $\vA$ to the following {\em peeling process}:
	\begin{quote}
	while there is a column $j$ with either no or just one non-zero entry, remove that column and, if it had exactly one non-zero entry, the row where that entry occurs from the linear system.
	\end{quote}
Write $\vA_* x_*=\vy_*$ for the resulting (possibly empty) reduced linear system.
By construction any solution to the original system $\vA x=\vy$ projects down to a solution to the reduced system.
Since the peeling process corresponds precisely to the process that determines the $2$-core of a random $k$-uniform hypergraph~\cite{Mike},
we know the distribution of the number of rows and columns of $\vA_*$, and $d_k$ is precisely the critical value from where the number of rows exceeds the number of columns.
In effect, since the right-hand side $\vy_*$ is chosen uniformly and independently of $\vA_*$, the reduced system is unlikely to have a solution for $d>d_k$, and thus the same is true of the original system.
Moreover, since the size of the reduced matrix $\vA_*$ only depends on the positions of the non-zero entries, $d_k$ is independent of $q$ and $P$.
Hence, \Thm~\ref{Thm_SAT} shows that the universal upper bound $d_k$ on the satisfiability threshold is universally tight.
Of course, for $q=2$ the formula from \Thm~\ref{Thm_SAT} matches the known $k$-XORSAT threshold.
	\footnote{To compare (\ref{eqdk}) with the formula given for the $k$-XORSAT threshold in~\cite[\Cor~1]{Dietzfelbinger}, we express the quantities $c,\beta,p_*$ from~\cite{Dietzfelbinger} in terms of our parameters $d,\rho_{k,d}$ as $p_*=\rho_{k,d}$, $\beta=d\rho_{k,d}^{k-1}$ and $c=d/k$.}

Since the solutions to the linear system $\vA x=\vy$ are affine translations of the kernel and $\rk(\vA)+\nul(\vA)=n$,
(\ref{eqrank}) implies immediately that with probability tending to one as $n\to\infty$, the number of solutions to the system $\vA x=\vy$ is exactly equal to $q^{n-\vm}$.
\aco{Actually the proof of \Thm~\ref{Thm_SAT} also yields the rank of the random matrix above the satisfiability threshold.}

\begin{corollary}\label{Cor_rank}
\aco{Let $k\geq3$, let $q>1$ be a prime power and let $P$ be a permutation-invariant distribution on $\FF_q^{*\,k}$.
Then for any $d>d_k$,
	\begin{align*}
	\frac{\rk(\vA)}{\vm}&\quad\stacksign{$n\to\infty$}\longrightarrow\quad1+\frac kd\rho_{k,d}-k\rho_{k,d}^{k-1}+(k-1)\rho_{k,d}^k\qquad\mbox{in probability}.
	\end{align*}}
\end{corollary}

\noindent
\aco{In the case $q=2$ this result was recently obtained independently by Cooper, Frieze and Pegden~\cite{CFP}, using a different approach.}

Returning to the satisfiable regime $d<d_k$,
the next two results concern the geometry of the set $\cS(\vA,\vy)$  of solutions to the random linear system, a question that
has been studied extensively in the case $q=2$~\cite{AchlioptasMolloy,Ibrahimi,MRTZ}.
First, what is the relative location of two randomly chosen solutions?
Let us define the {\em overlap} of $x,x'\in\FF_q^n$ as the $q\times q$-matrix $\omega(x,x')=(\omega_{\sigma,\tau}(x,x'))_{\sigma,\tau\in\FF_q}$ with entries
	$$\omega_{\sigma,\tau}(x,x')=\frac1n\sum_{i=1}^n\vecone\{x_i=\sigma,x_i'=\tau\}.$$
Thus, $\omega(x,x')$ is a probability distribution on $\FF_q\times\FF_q$, namely the empirical distribution of the value combinations of the $n$ coordinates.
Further, let us write $\vec x,\vec x'$ for two uniformly and independently chosen solutions of the linear system $\vA x=\vy$, and write
 $\bck{\nix}_{\vA,\vy}$ for the average over $\vec x,\vec x'$.
Finally, let $\bar\omega$ be the uniform distribution on $\FF_q\times\FF_q$ and
 let us write $\tv\nix$ for the total variation norm.

\begin{theorem}\label{Thm_overlap}
Let $k\geq3$, let $q>1$ be a prime power,  let $P$ be a permutation-invariant distribution on $\FF_q^{*\,k}$ and assume that $d<d_k$.
Then
	\begin{align*}
	\lim_{n\to\infty}\Erw\brk{\bck{\tv{\omega(\vec x,\vec x')-\bar\omega}}_{\vA,\vy}\,\big
		|\,\exists x\in\FF_q^n:\vA x=\vy}&=0.
	\end{align*}
\end{theorem}

\noindent
\Thm~\ref{Thm_overlap} shows that right up to the satisfiability threshold, from a macroscopic point of view pairs of random solutions `look uncorrelated'.
Formally, they have the same (uniform) overlap as two points $\vec y,\vec y'$ that are just chosen uniformly and independently from the entire space $\FF_q^n$.
This is in contrast to other constraint satisfaction problems such as random graph $3$-colouring, where the overlap distribution undergoes a phase transition and becomes non-uniform strictly before the satisfiability threshold~\cite{CKPZ}.
Indeed, in statistical physics jargon~\cite{pnas} \Thm~\ref{Thm_overlap} shows that random linear equations do not undergo a {\em condensation phase transition}.

While \Thm~\ref{Thm_overlap} takes a global perspective, we next show that prior to the threshold $d_k$ the geometry of the set $\cS(\vA,\vy)$ undergoes an interesting local transition.
Let $x,x'\in\cS(\vA,\vy)$ be solutions to the linear system.
Call $x$ {\em reachable} from $x'$ if there exists a sequence $x^{(0)},\ldots,x^{(\ell)}$ of solutions with
$x^{(0)}=x$, $x^{(\ell)}=x'$ such that the Hamming distance of $x^{(i)}$ and $x^{(i+1)}$ is bounded by $(\ln n)\cdot(\ln\ln n)$ for all $i<\ell$.%
\footnote{The function $(\ln n)\cdot(\ln\ln n)$ could be replaced by any function that grows asymptotically faster than $\ln n$ as $n\to\infty$.}
Hence, we can walk from $x$ to $x'$ within the set of solutions by only changing the values of only a small number of variables at each step.
Clearly, reachability is an equivalence relation.
We call the equivalence classes the {{\em solution clusters}} of the linear system and denote by $\Sigma(\vA,\vy)$ the set of all solution clusters.
The {\em distance} of two clusters $S,S'\in\Sigma$ is defined as the minimum Hamming distance between solutions $x\in S$ and $x'\in S'$.
Finally, we say that the random linear system enjoys a property {\em with high probability} (`\whp') if the probability that the property holds tends to one as $n\to\infty$.

\begin{theorem}\label{Cor_freezing}
Let $k\geq3$, let $q>1$ be a prime power and let $P$ be a permutation-invariant distribution on $\FF_q^{*\,k}$.
Further, with $\rho_{k,d}$ from \eqref{eqrho} let
	\begin{align}\label{eq2core}
	d_{k}^\star&=\inf\cbc{d>0:\rho_{k,d}>0}.
	\end{align}
\begin{enumerate}[(i)]
\item If $d<d_k^\star$, then the solutions to the linear system $\vA x=\vy$ form a single cluster \whp
\item If $d_k^\star<d<d_k$, then
	\begin{align}\label{eqCor_freezing}
	\frac1n\ln|\Sigma(\vA,\vy)|&\quad\stacksign{$n\to\infty$}\longrightarrow\quad \bc{\rho_{k,d}-d\rho_{k,d}^{k-1}+d(1-1/k)\rho_{k,d}^k}\ln q>0
		&\mbox{in probability}.
	\end{align}
	Moreover, the minimum distance between distinct solution clusters is $\Omega(n)$ \whp
\end{enumerate}
\end{theorem}

Thus, while for $d<d_k^\star$ it is possible to walk within the set $\cS(\vA,\vy)$ from one solution to any other by changing at most $(\ln n)(\ln\ln n)$ variables at a time, for $d>d_k^\star$ the set of solutions shatters into an exponential number of clusters \whp\ that are separated by a linear Hamming distance.
Once more the threshold $d_{k}^\star$ has a combinatorial interpretation.
It is the smallest value of $d$ from where the aforementioned peeling process leaves us with a non-empty reduced system $\vA_*x_*=\vy_*$  with high probability.
In fact, the proof of \Thm~\ref{Cor_freezing} reveals that 
the clusters and the solutions to the reduced linear system are essentially in one-to-one correspondence.
In statistical physics jargon~\cite{pnas}, \Thm~\ref{Cor_freezing} shows that random linear equations undergo a {\em dynamical phase transition} at $d_k^\star$.
The case $q=2$  was previously established by Achlioptas and Molloy~\cite{AchlioptasMolloy} and independently by Ibrahimi et al.~\cite{Ibrahimi}.

As an application of \Thm~\ref{Thm_SAT}
we obtain a result on the satisfiability of random equations over finite Abelian groups.
Thus, let $(\Gamma,+)$ be a finite Abelian group of order greater than one.
Then $\Gamma$ gives rise to a natural generalisation of the random $k$-XORSAT problem.
Namely, with $(\vec\alpha_{i,j})_{i,j}$ as in the definition of the random matrix $\vA$
and with a uniformly and independently chosen $\vy\in\Gamma^{\vm}$, we may ask whether it is possible to simultaneously satisfy
the random equations $\sum_{h=1}^kx_{\vec \alpha_{i,h}}=\vec y_i$.
The following theorem shows that the satisfiability threshold for this problem coincides with $d_k$ from (\ref{eqdk}), for any {finite Abelian} group $\Gamma$.

\begin{theorem}\label{Thm_Abelian}
Let $k\geq3$ and let $(\Gamma,+)$ be a finite Abelian group of order greater than $1$.
Then
	\begin{align*}
	\lim_{n\to\infty}\pr\brk{\exists x\in\Gamma^n:\bigwedge_{i=1}^{\vm}\bc{\sum_{h=1}^kx_{\vec \alpha_{i,h}}=\vec y_i}}
		&=\begin{cases}
			1&\mbox{ if }d<d_k,\\
			0&\mbox{ if }d>d_k.
			\end{cases}
	\end{align*}
\end{theorem}

\subsection{Related work}\label{Sec_related}
Over the past couple of decades the hunt for satisfiability thresholds has emerged to be remarkably challenging and
the ensuing work has uncovered surprising connections with other areas of mathematics, particularly mathematical physics.
The oldest specimen of a satisfiability threshold problem appears in the paper of \Erdos\ and \Renyi\ that started the theory of random graphs~\cite{ER60}.
One of the many problems that they posed in that paper asks for the $q$-colorability threshold in random graphs, i.e., the largest edge density up to which the chromatic number of the random graph with $n$ vertices and $m$ edges is bounded by $q$.
This is the single problem from that seminal paper that remains open to this day.

Indeed, to date the exact satisfiability thresholds are known in just a handful of problems.
The first example was the $2$-SAT threshold
(satisfiability of random Boolean formulas in conjunctive normal form with two variables per clause), pinpointed independently by Chvatal and Reed~\cite{mick} and Goerdt~\cite{Goerdt} in 1992.
The proof exploits the peculiar structure of the $2$-SAT problem and is similar in spirit to a percolation argument.
An analogous percolation argument actually also solves the problem of random linear equations with $k=2$ variables per equation~\cite{Kolchin}.
The satisfiability threshold in another variant of the satisfiability problem, the 1-in-$k$-SAT problem, was determined by Achlioptas et al.~\cite{ACIM} via the analysis of an algorithm by the method of differential equations.

In all other examples where the satisfiability threshold is known the proof is by way of the second moment method, 
first applied to a variant of the $k$-SAT problem by Frieze and Wormald~\cite{FriezeWormald}, then
used simultaneously by Dubois and Mandler~\cite{DuboisMandler} in their work on  $3$-XORSAT and by Achlioptas and Moore~\cite{nae} in their pathbreaking paper on the approximate location of the $k$-SAT threshold for $k\geq3$.
Currently second moment based proofs are known for the precise location of the $k$-SAT threshold~\cite{DSS3}, the $k$-NAESAT threshold in formulas with regular variable degrees~\cite{DSS1} and the satisfiability threshold in $k$-SAT formulas with regular literal degrees~\cite{KostaSAT},
always under the assumption that $k$ exceeds some large constant.
Additionally, the second moment method has been extended to pinpoint the satisfiability thresholds in 
certain specific examples of `uniquely extendible' constraint satisfaction problems~\cite{Connamacher,GoerdtFalke}, where, similarly as in linear equations, given the values of the first $k-1$ variables in a constraint there is precisely one way of setting the last variable so that the constraint is satisfied.
Moreover,  pinpointing the exact threshold for random regular positive $1$-in-$k$-SAT, a.k.a. Exact Cover, Moore~\cite{Moore}
provided evidence supporting a conjecture from~\cite{LM} that the second moment method is tight for ``locked'' problems.

A question for future work might be whether the present proof technique can be extended to derive the satisfiability thresholds of {\em all} uniquely extendible problems, or at least of the `locked' problems defined in~\cite{LM}.

\aco{The case of denser random matrices over finite fields with a (not too slowly) diverging average number of non-zero entries per row is significantly simpler, as is the $k=2$ of precisely two non-zero entries per row.
The literature on these cases is surveyed in~\cite{Kolchin}.}

Based on intuition developed in the study of disordered physical systems such as glasses, physicists developed an analytic but non-rigorous method for the study of sparse random structures called the {\em cavity method}~\cite{MM}.
The aforementioned derivations of the thresholds in random $k$-SAT and $k$-NAESAT not only verify the physics `predictions' in these specific cases but the proofs also extensively harness combinatorial insights gained in the physics work.
From the physics viewpoint the $k$-XORSAT problem serves as a template of a generic random constraint satisfaction problem~\cite[\Chap~18]{MM}.
In particular, a core tenet of the physics deliberations is that a decomposition of the set of solutions into clusters as provided by \Thm~\ref{Cor_freezing} is a universal feature of random constraint satisfaction problems.
This hypothesis provides the basis for the Survey Propagation message-passing scheme, upon which the physics calculations of satisfiability thresholds hinge~\cite{MPZ}.
Indeed, it was observed shortly after the work of Dubois and Mandler that the Survey Propagation technique `predicts' the correct satisfiability threshold in random $k$-XORSAT~\cite{CDMM,MRTZ} and
in~\cite{MRTZ} the Survey Propagation calculations are compared in detail with the second moment argument for general $k\geq3$, which is not, however, carried out analytically in full detail.

The random $k$-XORSAT problem is conceptually simpler than random $k$-SAT or random graph $q$-colouring because in the former the overlap between two randomly chosen solutions remains uniform right up to the satisfiability threshold.
\Thm~\ref{Thm_overlap} shows that random linear equations over the field $\FF_q$ share the same behaviour.
For $q=2$, Pittel and Sorkin~\cite{PittelSorkin} even showed that overlap concentration  holds inside the scaling window and thus established satisfiability of the reduced linear system $\vA_{*}x=\vy_*$ \whp\ whenever $m_*-n_*\gg1$.
By contrast, in random $k$-SAT and random graph $q$-colouring a {condensation phase transition} has been predicted to occur shortly before the actual satisfiability threshold~\cite{pnas}.
In the random graph colouring problem this conjecture has been verified rigorously~\cite{CKPZ}.
This phase transition marks the onset of intricate long-range correlations that 
spoil the concentration of the overlap.
Indeed, avoiding these long-range correlations is a major technical hurdle in the aforementioned work on the $k$-SAT threshold~\cite{KostaSAT,DSS3}.
As we shall see, the absence of a condensation phase transition in random linear equations is a subtle consequence of the fact that the set of solutions consists of affine translations of the kernel of the matrix $\vA$.

Beyond their intrinsic combinatorial interest, random constraint satisfaction problems have come to play a role in several other disciplines.
For instance, the random $k$-SAT problem plays an important role in computer science as a benchmark for satisfiability algorithms~\cite{Cheeseman,KirkpatrickSelman,MitchellSelmanLevesque} as well as in lower bounds in proof complexity (e.g.,~\cite{BenSasson}).
Moreover, random constraint satisfaction problems have been seized upon as gadgets in hardness proofs in computational complexity theory
	(e.g.,~\cite{Andreas,SlyUniqueness,VV}).
Additionally, low-density parity check (`LDPC') and low-density generator matrix (`LDGM') codes, based on random linear systems over finite fields, are the mainstay of modern coding theory~\cite{RichardsonUrbanke}.
Finally, there are intriguing connections to the statistical mechanics of disordered systems~\cite{MM}, and there are physics-inspired applications to a panoply of Bayesian inference problems~\cite{LF}.

\subsection{Preliminaries and notation}\label{Sec_prelims}
Throughout the paper we let $q\geq2$ be a prime power, $k\geq3$ an integer, $P$ a permutation-invariant probability distribution on $\FF_q^{*\,k}$
and $n\geq k$ an integer.
Furthermore, we always let $\vm$ be a random variable with distribution $\Po(dn/k)$ and $\vA$ the random $\vm\times n$ matrix as described above.
Sometimes we write $\vA_n=\vA$ to make the dependence on $n$ explicit.
Moreover, for an integer $m\geq0$ we write $\vA_{n,m}$ for the random matrix $\vA_n$ given that $\vm=m$.
As before $\vy\in\FF_q^{\vm}$ denotes a uniformly random vector, chosen independently of $\vA$ given $\vm$.

For an integer $l\geq1$ we use the shorthand $[l]=\{1,\ldots,l\}$.
Furthermore, we use the standard asymptotic symbols $o(\nix),O(\nix),\Omega(\nix)$ always with respect to the limit $n\to\infty$.
For instance, $o(1)$ stands for a term that tends to $0$ as $n\to\infty$.

We denote the rank, the nullity and the {kernel of an} $m\times n$-matrix $A$ by $\rk(A)$, $\nul(A)$ and $\ker(A)$, respectively.
Further, $\Im(A)$ denotes the image of $A$.
Additionally, we write $\cS(A,y)$ for the set of solutions to the linear system $Ax=y$ and
 $Z(A,y)=|\cS(A,y)|$ for the number of solutions.
Further, we let
	\begin{align}\label{eqZnullity}
	Z(A)&=Z(A,0)=q^{\nul(A)}.
	\end{align}

Suppose that $A$ is an $m\times n$-matrix over $\FF_q$.
We define a hypergraph $G(A)$ that describes the positions of the non-zero entries of $A$.
The vertex set of this hypergraph is the set $\{x_1,\ldots,x_n\}$.
Moreover, each row of $A$ induces a hyperedge.
Specifically, the hyperedge corresponding to the $i$'th row contains all $x_h$ such that the $(i,h)$'th entry of $A$ is non-zero.
Thus, $G(A)$ has at most $m$ hyperedges.
Furthermore, the hypergraph $G(\vA)$ of the random matrix $\vA$ is $k$-uniform (i.e., every hyperedge has $k$ vertices).

Remember that the {\em $2$-core} of a hypergraph $G$ is the largest sub-hypergraph in which every vertex has degree at least two (i.e., it appears in at least two hyperedges).
The $2$-core of a hypergraph can be constructed by iteratively removing vertices of degree zero or one along with, in the latter case, the single hyperedge in which the vertex appears.
This is precisely analogous to the aforementioned peeling process for obtaining the reduced linear system $(\vA_*,\vy_*)$.
In fact, the hypergraph $G(\vA_*)$ coincides with the 2-core of the hypergraph $G(\vA)$.
Hence, the following theorem determines the asymptotic size of the reduced matrix $\vA_*$.
Recall $\rho_{k,d}$ from (\ref{eqrho}) and $d_k^\star$ from (\ref{eq2core}).

\begin{theorem}[\cite{GaoMolloy}]\label{Thm_Mike}
The $2$-core of $\G(\vA)$ is empty \whp\ if $d<d_k^\star$.
Moreover, for $d>d_k^\star$ the number $n_*(\vA)$ of vertices and the number $m_*(\vA)$ of edges of the $2$-core of $G(\vA)$ satisfy
	\begin{align}
	n_*(\vA)/n&\ \stacksign{$n\to\infty$}\longrightarrow\ \rho_{k,d}-d\rho_{k,d}^{k-1}+d\rho_{k,d}^k\quad\mbox{and}\quad
	m_*(\vA)/n\ \stacksign{$n\to\infty$}\longrightarrow\ d\rho_{k,d}^k/k
		\qquad\mbox{in probability}. \label{nm}
	\end{align}
	{Moreover, the function
	\begin{equation}
	\pi_{k}: d\in[d_k^\star,\infty)\mapsto \frac dk\rho_{k,d}^k{\left( \rho_{k,d}-d\rho_{k,d}^{k-1}+d\rho_{k,d}^k\right)}^{-1} \label{monotone}
	\end{equation}
	 is strictly increasing.}
\end{theorem}
\begin{proof}
{
The parametrisation employed in~\cite{GaoMolloy} differs from the one used here, so we should point out how to get from one to the other.
The parameters $c,r$ in~\cite{GaoMolloy} can be expressed in terms of ours as $c=d/k$, $r=k$.
A convex function $h(\mu)=\mu/(1-\eul^{-\mu})^{k-1}$ was employed~\cite[eq.(3)]{GaoMolloy} to deduce $d_k^\star$.
If $k\ge 3$ then $h(\mu)$ tends to infinity as $\mu\to 0$ and as $\mu\to\infty$.
Consequently, there is a unique stationary point $\mu^\star$, which is the global minimum of $h(\mu)$, and $d_k^\star=h(\mu^\star)$.
If $d>d_k^\star$, then there are two solutions of the equation $h(\mu)=d$.
Let $\mu_{k,d}$ denote the larger one.
Using the transformation $x=1-\exp(-\mu)$ 
we obtain a one-to-one correspondence between the two solutions of $h(\mu)=d$ and the two positive fixed points of $x\in[0,1]\mapsto1-\exp(-dx^{k-1})$.
Moreover, $\mu_{k,d}$ corresponds to $\rho_{k,d}$ under this transformation, and \cite[\Lem~7]{GaoMolloy} says that $n_*(\vA)/n\to 1-(1+\mu_{k,d})\eul^{-\mu_{k,d}}$ and $m_*(\vA)/n\to \mu_{k,d}(1-\eul^{-\mu_{k,d}})/k$ in probability. Immediately this gives (\ref{nm}). 
 
 Again, using $\rho_{k,d}=1-\exp(-\mu_{k,d})$, we have $\pi_k(d)=g_k(\mu_{k,d})/k$, where
 \[
 g_k(x)=\frac{x(1-\eul^{-x})}{k(1-(1+x)\eul^{-x})}.
 \]
 Since $h(\mu)$ is an increasing function on $[\mu^\star,\infty)$, $\mu_{k,d}$ is an increasing function of $d$ on $[d_k^\star,\infty)$ (recall that $d_k^\star=h(\mu^\star)$). By~\cite[\Lem~27]{GaoMolloy}, $g_k(x)$ is an increasing function on $(0,\infty)$. It follows then that 
$\pi_k(d)$ is an increasing function on $[d_k^\star,\infty)$.}
\end{proof}
\noindent
As an immediate consequence  we obtain the upper bound on the satisfiability threshold for \Thm~\ref{Thm_SAT}.

\begin{lemma}\label{Cor_Mike}
If $d>d_k$, then $\lim_{n\to\infty}\pr[\exists x\in\FF_q^n:\vA x=\vy]=0$.
\end{lemma}
\begin{proof}
Since any solution to $\vA x=\vy$ induces a solution to $\vA_*x_*=\vy_*$, it suffices to prove that the reduced system does not have a solution \whp\
\Thm~\ref{Thm_Mike} shows that together with the choice \eqref{eqdk} of $d_k$, for any $d>d_k$ there exists $\eps>0$
such that $m_*(\vA)>n_*(\vA)+\eps n$ \whp\
Hence, the matrix $\vA_*$ has rank at most $m_*(\vA)-\eps n$ \whp\, and 
if so, then the probability that the random vector $\vy_*$ belongs to the image of $\vA_*$ is upper-bounded by $q^{-\eps n}$.
\end{proof}

Finally, a substantial part of the paper deals with the analysis of the uniform distribution on the set of solutions of a linear system.
More generally, for a finite set $\Omega\neq\emptyset$ let $\cP(\Omega)$ denote the set of probability distributions on $\Omega$.
Hence, for an integer $n\geq1$ we write $\cP(\Omega^n)$ for the set of probability distributions on the discrete cube $\Omega^n$.
Probability measures on discrete cubes are well understood and we shall use a few known facts and concepts about them from~\cite{Victor}.
First, for $\mu\in\cP(\Omega^n)$ and a set $I\subset[n]$ we denote by $\mu_I\in\cP(\Omega^I)$ the joint distribution of the coordinates in $I$.
That is,
	$$\mu_I(\sigma)=\sum_{\tau\in\Omega^n}\vecone\cbc{\forall i\in I:\sigma_i=\tau_i}\mu(\tau)\qquad\qquad\mbox{for }\sigma\in\Omega^I.$$
We normally write $\mu_{i_1,\ldots,i_k}$ instead of $\mu_{\{i_1,\ldots,i_k\}}$.
Further, the measure $\mu$ is {\em $(\eps,l)$-symmetric} if
	\begin{align*}
	\sum_{1\leq i_1<\cdots<i_l\leq n}\TV{\mu_{i_1,\ldots,i_l}-\mu_{i_1}\tensor\cdots\tensor\mu_{i_l}}<\eps n^l.
	\end{align*}
Intuitively this means that for most sets $\{i_1,\ldots,i_l\}$ of coordinates the joint distribution $\mu_{i_1,\ldots,i_l}$ is close to the
product distribution with the same marginals.
The following basic fact shows that in order to establish $(\delta,l)$-symmetry, it suffices to prove $(\eps,2)$-symmetry for a suitably small $\eps=\eps(\delta,l)>0$.

\begin{lemma}[\aco{\cite[Corollaries~2.3 and 2.4]{Victor}}]\label{lem:k-wise}
For any $\delta>0$, $l\geq3$, $\Omega\neq\emptyset$ there exists $\eps=\eps(\delta,k,\Omega)>0$ such that for all $n>1/\eps$ the following is true.
If $\mu\in\cP(\Omega^n)$ is $(\eps,2)$-symmetric, then  $\mu$ is $(\delta,l)$-symmetric.
\end{lemma}

\section{The proof strategy}\label{Sec_2}

\noindent
In this section we outline the proof of the main results.
In contrast to prior work~\cite{Dietzfelbinger,DuboisMandler,GoerdtFalke}, we do not proceed by calculating the
	 second moment of the number $Z(\vA,\vy)$ of solutions to the linear system.
Instead, we pursue an approach inspired by coding theory.
Nonetheless, it is instructive to take a glimpse at the perils of the second moment approach first.

\subsection{Avoiding the second moment method}\label{Sec_techniques}
Why is the second moment argument for $k$-XORSAT, let alone for random linear equations over $\FF_q$, so tricky?
Actually it is easy enough to derive the explicit formula for the second moment for $q=2$ (i.e., for random $k$-XORSAT).
For starters, the first moment works out to be
	\begin{align}\label{eqvanilla1st}
	\Erw[Z(\vA,\vy)|\vm]=\sum_{x\in\FF_2^n}\pr\brk{\vA x=\vy|\vm}=2^{n-\vm};
	\end{align}
for there are $2^n$ candidate solutions $x$ and each satisfies $\vA x=\vy$ with probability $2^{-\vm}$  because $\vy$ is chosen  independently of $\vA$.
Further, the second moment is just the expected number of pairs of solutions.
Hence, we can easily write an explicit expression for the second moment as well because
 the probability that both $x,x'\in\FF_2^n$ satisfy the linear system depends on the number $s$ of entries 
on which the two vectors agree  only.
In fact, a simple inclusion/exclusion argument shows that the probability is equal to $[(1+\bc{(2s/n)-1}^k)/4]^{\vm}$.
Since the number of pairs $x,x'$  that agree on $s$ coordinates equals $2^n\bink ns$, we obtain
	\begin{align}\label{eqsmm}
	\Erw[Z(\vA,\vy)^2|\vm]=\sum_{x,x'\in\FF_2^n}\pr\brk{\vA x=\vA x'=\vy|\vm}=\sum_{s=0}^n2^n\bink ns\bcfr{1+\bc{(2s/n)-1}^k}4^{\vm}.
	\end{align}
Ideally we would like to show that $\Erw[Z(\vA,\vy)^2|\vm]\sim\Erw[Z(\vA,\vy)|\vm]^2$.
Then Chebyshev's inequality would reveal that $Z(\vA,\vy)\sim\Erw[Z(\vA,\vy)|\vm]=2^{n-\vm}$ \whp, whence $\vA$ has rank $\vm$ \whp

\begin{figure}
\includegraphics[height=4cm]{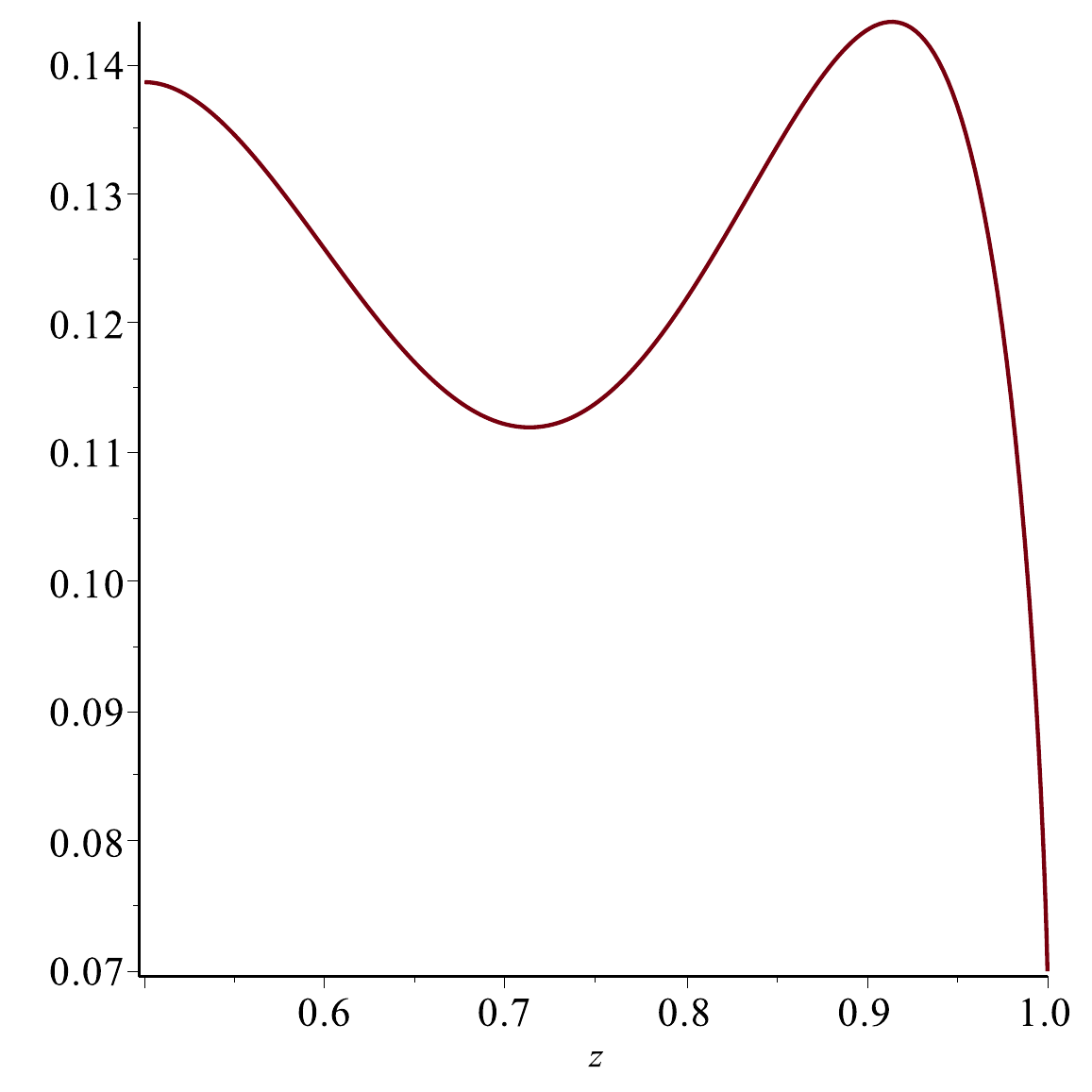}
\caption{The function from (\ref{eqlogsmm}) for $k=3$ and $d=2.7$ for $1/2\leq z\leq1$.
	The maximum is attained at $z\approx0.91$ rather than $z=1/2$.}\label{Fig_XORSATplot}
\end{figure}

A bit of calculus shows that $\Erw[Z(\vA,\vy)^2|\vm]\sim\Erw[Z(\vA,\vy)|\vm]^2$ iff the sum in (\ref{eqsmm})  is dominated by the contribution from the terms $s\sim n/2$.
This is indeed the case for small values of $\vm$.
But for $d$ below but near the satisfiability threshold significantly larger values of $s$ take over.
To see this, we take logarithms to obtain from (\ref{eqsmm}) that for $\vm\sim dn/k$, with the transformation $z=s/n$,
	\begin{align}\label{eqlogsmm}
	\frac1n\ln\Erw[Z(\vA,\vy)^2|\vm]=\max_{z\in[0,1]}-z\ln z-(1-z)\ln(1-z)+\frac{d}k\ln(1+(2z-1)^k)+(1-2d/k)\ln 2+o(1).
	\end{align}
Hence, the second moment argument can succeed only if the maximum 
is attained at $z=1/2$.
But for example for $k=3$ this ceases to be true for $d>2.7$, while the actual satisfiability threshold is $d_3\approx2.75$; see Figure~\ref{Fig_XORSATplot}.

The explanation is that the number $Z(\vA,\vy)$ of solutions is subject to a lottery effect:
the second moment is driven up by a minority of solution-rich linear systems.
Combinatorially this is due to large deviations of 
the size of the reduced linear system $\vA_*x_*=\vy_*$ obtained via the peeling process.
Indeed, the hump near $0.91$ in Figure~\ref{Fig_XORSATplot} stems from linear systems whose reduced system has a solution but where the $2$-core of the hypergraph $G(\vA)$ is much denser than expected.
In effect, even a single solution to the reduced system may extend in more than $2^{n-\vm}$ ways to a solution of the entire system $\vA x=\vy$, causing a very large contribution to the second moment.

Dubois and Mandler~\cite{DuboisMandler} and the subsequent contributions~\cite{Dietzfelbinger2,PittelSorkin} applied the obvious remedy,
which is to condition on the number of rows and columns of the reduced system.
But unfortunately this conditioning leaves us with a second moment formula {\em far} more complicated than (\ref{eqsmm}).
Among other things, it now becomes necessary to not just count how many variables take the same values under $x,x'$ but also 
how many non-zero entries there are in the  corresponding columns.
Unsurprisingly, showing that the dominant contribution still comes from $x,x'$ that agree on $s\sim n/2$ entries right up to the actual threshold $d_k$ is anything but straightforward.

In fact, as highlighted by the technically impressive work of Falke and Goerdt~\cite{GoerdtFalke}, these difficulties are exacerbated significantly if $q>2$.
For instead of having to optimise a single parameter as in the case $q=2$,
we now end up with an intricate non-convex $(q-1)^2$-dimensional optimisation problem that asks for the $q\times q$-overlap matrix rendering the largest contribution to the second moment.
As underscored by work on random graph/hypergraph colouring problems~\cite{AchNaor,DFG},
tackling such high-dimensional optimisation tasks is notoriously difficult.

\subsection{From inference to satisfiability}\label{Sec_strategy}
We therefore pursue a very different proof strategy.
We approach the satisfiability threshold problem from the viewpoint of an inference or decoding task.
More specifically, we think of the matrix $\vA$ as the generator matrix of a linear code.
Thus, the encoding of a `message' $x\in\FF_q^n$ is simply the vector $\vhy=\vA x$.
Normally in a coding setting $\vhy$ would get transmitted along a noisy channel and we would ask how many bits of $x$ it is possible to reconstruct from the noisy observation.
In fact, this is precisely the bitwise maximum a-posteriori inference problem for low-density generator matrix codes~\cite{KabashimaSaad,McEliece}.

The present proof strategy rests on the observation that the satisfiability threshold problem boils down to the LDGM decoding problem {\em without} noise.
Hence, we ask how much information about $x$ it is possible to reconstruct from the vector $\hat y=\vA x$ itself.
Clearly, each entry $\hat y_i$ is a sparse observation of the vector $x$, namely, the inner product of $x$ and a random vector with $k$ non-zero entries.
Consequently, if $d$ is quite small,  then  the observations $\hat y$ will be very sparse and very dispersed and in effect we cannot hope to recover $x$ well.
But as the number of rows increases, i.e., as we acquire more observations of $x$, there will eventually occur a threshold from where
the recovery of an extensive fraction of bits becomes information-theoretically possible. That is, after the information-theoretic threshold, there exists an absolute constant $\eps>0$ such that with probability tending to 1 as $n$ goes to infinity, at least $\eps n$ coordinates of $x$ can be recovered correctly given $\hat y$.
We shall see that this threshold can be determined precisely, and that it coincides with the satisfiability threshold.
In summary, our task comes down to identifying the {\em exact} LDGM decoding threshold in the absence of noise.

Motivated by data compression problems, Wainwright et al.~\cite{WMM} studied this decoding task in the case $q=2$.
They also noticed the connection with the $k$-XORSAT threshold.
However, reliant on the second moment method, their analysis was not tight and 
did not yield the precise $k$-XORSAT threshold.
Moreover, the LDGM decoding problem {\em with noise} was recently studied precisely~\cite{CKPZ}.
We are going to seize upon certain ideas from that proof, particularly the use of the `Aizenman-Sims-Starr scheme'.
However, we cannot follow the LDGM argument directly for two reasons.
First, that proof critically depends on a positive noise parameter to turn the parity constraints into `soft'  constraints.
Second, even if the approach from~\cite{CKPZ} would work with zero noise, the resulting formula 
for the satisfiability threshold would come as an infinite-dimensional optimisation problem over the space of {all} probability measures on the $(q-1)$-dimensional simplex $\cP(\FF_q)$.
Proving that the optimal solution matches the innocuous expression from \Thm~\ref{Thm_SAT} appears to be a  formidable task.

On the positive side, the `zero noise' decoding problem possesses algebraic structure.
Diligently exploiting this will allow us to answer the decoding question precisely in a fairly elegant, transparent manner,
without having to engage high or even infinite-dimensional optimisation problems.
The following paragraphs outline the proof.

\subsubsection{The Nishimori identity}
Let $\vhx\in\FF_q^n$ be a uniformly random vector, chosen independently of $\vA$, and let $\vhy=\vA\vhx$.
We aim to figure out how many entries of $\vhx$ it is possible to reconstruct from $\vA$ and $\vhy$.
As a first step, what is the posterior distribution of $\vhx$ given $\vA$ and $\vhy$?

\begin{lemma}\label{Fact_Nishimori}
\begin{enumerate}
\item Given $\vA,\vhy$ the vector $\vhx$ is uniformly distributed over the set $\cS(\vA,\vhy)$.
\item Given $\vA$ the vector $\vhy$ is distributed uniformly over the set $\Im(\vA)$.
\item The joint distribution of $\vA,\vhy$ satisfies 
	\begin{align}\label{eqNishimori}
	\pr\brk{\vA=A,\vhy=y|\vm=m}&=
		\frac{Z(A,y)}{\Erw[Z(\vA,\vy)\mid\vm=m]}\cdot {q^{-\vm}}\,\pr\brk{\vA=A|\vm=m}
		&\mbox{for any }A\in\FF_q^{m\times n},\ y\in\FF_q^m.
	\end{align}
\end{enumerate}
\end{lemma}

Thus, the probability that the experiment of first choosing $\vA$, independently choosing $\vhx$ and setting $\vhy=\vA\vhx$ produces the linear system $(A,y)$ is proportional to its number $Z(A,y)$ of solutions.
In other words, the distribution $(\vA,\vhy)$ is obtained by `size-biasing' the distribution $(\vA,\vy)$ according to the number of solutions.
Following~\cite{CKPZ}, we call (\ref{eqNishimori}) the `Nishimori identity'.%
	\footnote{Strictly speaking, in statistical physics jargon the term `Nishimori property' refers a weaker identity
		that pertains to averages on $\vA$ and independent random solutions $\vec x^{(1)},\ldots,\vec x^{(\ell)}$ to
		the random linear system $\vA x=\vhy$.
		This weaker Nishimori identity follows immediately from \Lem~\ref{Fact_Nishimori}.}
Such identities have come to play a major role in the study of inference tasks such as decoding~\cite{CKPZ}.
The simple proof of \Lem~\ref{Fact_Nishimori} can be found in \Sec~\ref{Sec_Fact_Nishimori}.

For what values of $d$ can we hope to recover a non-trivial approximation to $\vhx$ from the linear system $\vA,\vhy$?
Since by \Lem~\ref{Fact_Nishimori} given $\vA,\vhy$ our best guess as to $\vhx$ is a uniformly random solution of the linear system, we cannot hope to reconstruct $\vhx$ if the linear system possesses a large number of `uncorrelated' solutions.
More formally, if the overlap $\omega(\vx,\vx')$ of two random solutions $\vx,\vx'$ is close to the uniform overlap $\bar\omega$, then the linear system $\vA,\vhy$ possesses an abundance of solutions, well spread over the entire space $\FF_q^n$, and $\vhx$ may be any one of them.
Thus, reconstructing $\vhx$ will be impossible if $d>0$ is such that
	\begin{align}\label{eqOverlap1}
	\lim_{n\to\infty}\Erw\brk{\bck{\tv{\omega(\vec x,\vec x')-\bar\omega}}_{\vA,\vhy}}&=0,
	\end{align}
	recalling that $\bck{\cdot}_{\vA,\vhy}$ denotes the average over uniform and independent solutions $\vx, \vx'$, given $\vA$ and $\vhy$.
If, on the other hand, $\tv{\omega(\vec x,\vec x')-\bar\omega}$ is bounded away from $0$, then the solutions to the linear system cannot be spread uniformly. Using correlations between solutions, a substantial share of the entries of $\vhx$ can be recovered by observing an arbitrary solution $x$. Our Theorem~\ref{Thm_overlap} implies that~(\ref{eqOverlap1}) holds all the way until the satisfiability threshold $d_k$. For $d>d_k$, the reduced linear system $(\vA_*,\vy_*)$ has more equations than variables by Theorem~\ref{Thm_Mike} and thus has a unique solution with high probability. Hence, we can recover all entries of $\vhx$ lying in the reduced system. 

How is this inference problem related to the satisfiability threshold of the linear system $\vA x=\vy$?
The whole problem with the `vanilla' second moment approach is to identify the overlap value that renders the dominant contribution.
But due to the Nishimori identity~(\ref{eqNishimori}), the {\em a priori} overlap estimate (\ref{eqOverlap1}) allows us to sidestep this issue completely by setting up a 
truncated second moment calculation where we only count pairs of solutions with an overlap close to $\bar\omega$.
Interestingly, to make this argument work we do not need to do anything `clever' (like referring to a reduced linear system or the $2$-core) at all, 
rather this truncated second moment calculation just comes down to a straightforward application of the Laplace method.
We will carry this out in detail in \Sec~\ref{Sec_Prop_smm}, where we establish the following.

\begin{proposition}\label{Prop_smm}
Assume that (\ref{eqOverlap1}) is satisfied for $d>0$.
Then the linear system $\vA x=\vy$ has a solution with high probability and thus
	$\lim_{n\to\infty}\pr\brk{\rk(\vA)=\vm}=1.$
\end{proposition}

\noindent
A second moment argument for an overlap-truncated variable was previously applied to a different family of models in~\cite{CEJKK}.
But there the method had to be combined with the small subgraph conditioning technique from~\cite{RobinsonWormald}.
Here matters are simpler, and a plain application of the Laplace method suffices, thanks to the algebraic nature of the problem.

\subsubsection{Random homogeneous systems}
Hence, we need to show that (\ref{eqOverlap1}) is valid for all $d<d_k$.
The starting point for this is provided by the following general fact concerning the uniform distribution on solutions to an arbitrary (not-necessarily-random) homogeneous linear system, whose
 proof can be found in \Sec~\ref{Sec_kernel}.

\begin{lemma}\label{Prop_kernel}
For any $A\in\FF_q^{m\times n}$
 there exists a decomposition $S_0(A), S_1(A),\ldots, S_{n}(A)$ of $[n]$ into pairwise disjoint (possibly empty) sets,
 unique up to a permutation of the classes $S_1(A),\ldots, S_{n}(A)$, with the following properties.
	\begin{enumerate}[(i)]
	\item For every $i\in S_0(A)$we have $\bck{\vecone\{\vx_i=0\}}_{A,0}=1$.
	\item If $i,j\in S_u(A)$ for some $u\geq1$, then
		$\bck{\vecone\{\vx_i=\sigma\}}_{A,0}=1/q$ for every $\sigma\in\FF_q$ and
		 there exists $s\in\FF_q^*$ such that 
		$$\bck{\vecone\{\vx_j=s\vx_i\}}_{A,0}=1.$$
	\item If $i\in S_u(A)$, $j\in S_v(A)$ for indices {$1\leq u<v  \leq n$}, then  for all $\sigma,\tau\in\FF_q$ we have
			$\bck{\vecone\{\vx_i=\sigma,\vx_j=\tau\}}_{A,0}=1/q^2$.
	\end{enumerate}
\end{lemma}

\noindent
Thus, for any vector $x\in\ker (A)$ all the coordinates $x_i$, $i\in S_0(A)$, are equal to zero.
Moreover, if $i,j$ belong to the same class $S_u(A)$, $u\geq1$, then $x_j$ is always some specific non-zero scalar multiple of $x_i$.
Finally, if $i,j$ belong to different classes $S_u(A),S_v(A)$, $u,v\geq1$, then in a random $\vx\in\ker(A)$ the entries $\vx_i,\vx_j$ are uniform and independent.

As an application of \Lem~\ref{Prop_kernel} we obtain the following sufficient condition for the overlap condition (\ref{eqOverlap1}).

\begin{corollary}\label{Prop_hom} 
Let $d>0$. The condition (\ref{eqOverlap1}) is satisfied if
	\begin{equation}\label{eqOverlap2}
	\lim_{n\to\infty}\frac1n\Erw[\max_{i\geq0}|S_i(\vA)|]=0.
	\end{equation}
\end{corollary}

\noindent
The proof of \Cor~\ref{Prop_hom} can be found in \Sec~\ref{Sec_hom}.
Thus, we are left to verify (\ref{eqOverlap2}) for all $d<d_k$.

\subsubsection{The nullity}
We are going to solve this task indirectly by first deriving an upper bound on the nullity of $\vA$.
Subsequently we will show that this bound could not possibly hold if (\ref{eqOverlap2}) were violated for some $d<d_k$.

\begin{proposition}\label{Prop_freeEnergy}
For all $d>0$ we have 
	\begin{align}\label{eqProp_freeEnergy}
	\limsup_{n\to\infty}\frac1n\Erw[\nul(\vA)]&\leq\sup_{\alpha\in[0,1]}\phi(\alpha),&\mbox{where}\\
	\phi(\alpha)&=\exp\bc{-d \alpha^{k-1}}+d\alpha^{k-1}-\frac{d(k-1)}{k}\alpha^k-\frac dk.\nonumber
	\end{align}
\end{proposition}
{\em Remark.} While~(\ref{eqProp_freeEnergy}) gives an upper bound for $\Erw[\nul(\vA)]$ for all $d>0$, if we consider  $d<d_k$, $\sup_{\alpha\in[0,1]}\phi(\alpha)=\phi(0)=1-d/k$, as shown in Lemma~\ref{Prop_opt} below. Hence, for $d<d_k$, Proposition~\ref{Prop_freeEnergy} compares the rank of $\vA$ with its number of rows.

The proof of \Prop~\ref{Prop_freeEnergy} is based on a coupling argument reminiscent of the Aizenman-Sims-Starr scheme from mathematical physics~\cite{Aizenman}.
This approach was used in prior work on LDGM codes~\cite{CKPZ}.
However, if we were to follow the previous Aizenman-Sims-Starr scheme directly, we would be left with an infinite-dimensional optimisation problem. 
Fortunately, once more the algebraic nature of the problem comes to the rescue.
Not only do we obtain a very clean coupling argument, but also can we confine the class of distributions that we need to optimise over to a one-dimensional set.

To be more precise, applied to our algebraic setting the Aizenman-Sims-Starr scheme argument comes down to the following simple observation.
Writing $\vA_n$ for the random linear system with $n\geq k$ variables and $\vm\sim\Po(dn/k)$ equations, we clearly have
	\begin{align}\label{eqAiz1}
	\Erw[\nul(\vA_n)]&=\Erw[\nul(\vA_k))]+{\sum_{N=k +1}^n}\Erw[\nul(\vA_N)]-\Erw[\nul(\vA_{N-1})].
	\end{align}
Consequently,
	\begin{align}\label{eqAiz1}
	\limsup_{n\to\infty}\frac1n\Erw[\nul(\vA_n)]&\leq\limsup_{n\to\infty}\Erw[\nul(\vA_{n+1})]-\Erw[\nul(\vA_n)].
	\end{align}
In order to estimate the difference $\Erw[\nul(\vA_{n+1})]-\Erw[\nul(\vA_n)]$ we will set up a coupling of the random systems $\vA_n,\vA_{n+1}$.
Roughly speaking, under this coupling $\vA_{n+1}$ is obtained from $\vA_n$ by adding one extra variable $x_{n+1}$ along with a few random linear equations,
and we need to estimate the ensuing change in nullity very carefully.
The proof of \Prop~\ref{Prop_freeEnergy}, which we carry out in \Sec~\ref{Sec_Prop_freeEnergy}, is the core of the entire paper.

\Prop~\ref{Prop_freeEnergy} leaves us with an elementary one-dimensional optimisation problem.
In \Sec~\ref{Sec_Prop_opt} we will solve this  problem  explicitly for $d<d_k$.

\begin{lemma}\label{Prop_opt} 
If $0<d\leq d_k$, then
	$\sup_{\alpha\in[0,1]}\phi(\alpha)=\phi(0)=1-d/k$.
\aco{Moreover, if $d>d_k$, then 
	$$\sup_{\alpha\in[0,1]}\phi(\alpha)=\phi(\rho_{k,d})=1-d/k-\rho_{k,d}+d\rho_{k,d}^{k-1}-d(k-1)\rho_{k,d}^k/k.$$}
\end{lemma}

{Although the reduced system $\vA_*x_*=\vy_*$ plays no explicit role in the proofs of \Prop~\ref{Prop_freeEnergy} and \Lem~\ref{Prop_opt}, the optimisation problem on the right hand side of (\ref{eqProp_freeEnergy}) is closely related to the nullity of $\vA_*$.
In fact, {the proof of} \Lem~\ref{Prop_opt} reveals that the supremum is attained at $\alpha=0$ precisely up to the degree $d_k$ where $\nul(\vA_*)=\Omega(n)$, while $\Erw[\nul(\vA_*)]=O(1)$ for $d>d_k$.
Furthermore, $\alpha$ is equal to the fraction of entries of the vector $\vhx$ that can be reconstructed from $\vA$ and $\vhy=\vA\vhx$, 
	up to algebraic symmetries (e.g., in the case $q=2$ and $k$ even it is impossible to tell $\vhx$ and its binary inverse $\vecone+\vhx$ apart).
Thus, $d_k$ does indeed mark the threshold from where non-trivial recovery of $\vhx$ becomes possible.
But let us not dwell on this because we won't need this fact to establish the main results.}

\subsubsection{Wrapping up}
In \Sec~\ref{Sec_wrap} we will deduce from \Prop~\ref{Prop_freeEnergy} and \Lem~\ref{Prop_opt} that (\ref{eqOverlap2}) is satisfied for all $d<d_k$.
The proof is based on assembling the random linear system $\vA$ with $\vm\sim\Po(dn/k)$ equations by adding one random equation at a time.
Computing the derivative $\frac{\partial}{\partial d}\frac1n\Erw[\nul(\vA)]$, 
we will show that if (\ref{eqOverlap2}) were violated for some $d<d_k$,
then the nullity would be much bigger than~(\ref{eqProp_freeEnergy}) permits.
{A similar argument was used in~\cite{CEJKK} for LDGM codes with positive noise in the case $q=2$, but it was necessary to assume that $k\geq4$ is even.
	Here the algebraic structure of the problem, properly exploited, allows us to deal with all $k\geq3$.}

\begin{proposition}\label{Prop_wrap}
{If (\ref{eqOverlap2}) fails to hold for some $d<d_k$, then there exists $D<d_k$ such that}
	$$\limsup_{n\to\infty}\frac1n\Erw[\nul(\vA_{n,\vm(D,n)})]>1-D/k.$$
\end{proposition}

\noindent
We are now ready to deduce \Thm s~\ref{Thm_SAT}--~\ref{Thm_Abelian} from the propositions above.

\begin{proof}[Proof of \Thm~\ref{Thm_SAT}]
\Lem~\ref{Cor_Mike} implies that there is no solution \whp\ if $d>d_k$.
On the other hand, 
\Prop~\ref{Prop_freeEnergy}, \Lem~\ref{Prop_opt} and \Prop~\ref{Prop_wrap} show that (\ref{eqOverlap2}) is satisfied for all $d<d_k$.
Thus, \Cor~\ref{Prop_hom} implies that the overlap condition (\ref{eqOverlap1}) holds for all $d<d_k$.
Therefore, \Prop~\ref{Prop_smm} shows that the random linear system $\vA x=\vy$ has a solution \whp\ for all $d<d_k$.
\Prop~\ref{Prop_smm} also yields the rank statement~\eqref{eqrank}.
\end{proof}

\begin{proof}[Proof of \Cor~\ref{Cor_rank}]
\aco{Since $\rk(\vA)=n-\nul(\vA)$, it suffices to calculate the nullity for $d>d_k$.
\Prop~\ref{Prop_freeEnergy} and \Lem~\ref{Prop_opt} yield an upper bound, namely
	\begin{align}\label{eqCor_rank}
	\frac1n\Erw[\nul(\vA)]&\leq1-d/k-\rho_{k,d}+d\rho_{k,d}^{k-1}-d(k-1)\rho_{k,d}^k/k+o(1).
	\end{align}
To obtain a matching lower bound, consider the modified linear system $\vA'x'=0$ where we replace every variable $x_i$ in the $2$-core of $G(\vA)$ by the constant value $0$.
Clearly, $\nul(\vA)\geq\nul(\vA')$ because every vector in the kernel of $\vA'$ extends to a vector in the kernel of $\vA$ by simply adding $0$ entries for all variables in the $2$-core of $G(\vA)$.
Furthermore, since all linear equations that comprise core variables only are trivially satisfied, the linear system $\vA'$ contains $\vm-m_*(\vA)$ equations.
Therefore, $\rk(\vA')\leq \vm-m_*(\vA)$.
Moreover, $n-n_*(\vA)$ variables remain, and thus $\nul(\vA')\geq n-n_*(\vA)-\vm+m_*(\vA)$.
Substituting the expression for $n_*(\vA),m_*(\vA)$ from \Thm~\ref{Thm_Mike}, we obtain a lower bound that matches (\ref{eqCor_rank}).}
\end{proof}

\begin{proof}[Proof of \Thm~\ref{Thm_overlap}]
Suppose $d<d_k$.
Then \Prop~\ref{Prop_freeEnergy}, \Lem~\ref{Prop_opt}, \Prop s~\ref{Prop_wrap} and \Cor~\ref{Prop_hom} imply that (\ref{eqOverlap1}) holds, and
by \Thm~\ref{Thm_SAT} the matrix $\vA$ has rank $\vm$ \whp\
If so, then the second part of \Lem~\ref{Fact_Nishimori} shows that the vector $\vhy$ is uniformly {distributed over $\FF_q^{\vm}$}, i.e.,
the linear systems $\vA x=\vy$ and $\vA x=\vhy$ are identically distributed.
Therefore, the assertion follows from (\ref{eqOverlap1}).
\end{proof}

\begin{proof}[Proof of \Thm~\ref{Cor_freezing}]
Achlioptas and Molloy~\cite{AchlioptasMolloy} proved \Thm~\ref{Cor_freezing} only for the special case $q=2$.
But their proof depends only on two ingredients:
knowledge of the satisfiability threshold, which \Thm~\ref{Thm_SAT} supplies for all $q$, and certain combinatorial properties of the
random hypergraph $G(\vA)$, which is independent of $q$.
Thus, thanks to \Thm~\ref{Thm_SAT} the argument of Achlioptas and Molloy carries over directly to all $q>2$.
The full details with the appropriate pointers to the specific statements from~\cite{AchlioptasMolloy} can be found in Appendix~\ref{Sec_freezing}.
\end{proof}

\begin{proof}[Proof of \Thm~\ref{Thm_Abelian}]
Any finite Abelian group $\Gamma$ of order greater than one can be written as a direct sum
	\begin{align*}
	\Gamma&=\bigoplus_{i=1}^K \ZZ/q_i\ZZ
	\end{align*}
of cyclic groups whose orders $q_1,\ldots,q_K\geq2$ are prime powers. It is therefore sufficient to show Theorem \ref{Thm_Abelian} for cyclic groups of the form $\Gamma=\ZZ_{p^\ell}$.

In the following, let thus $k \geq 3$, $p$ be a prime, $\ell \geq 1$ and $(\Gamma,+) = (\ZZ_{p^\ell}, +)$. 
We denote by $\vB \in \{0,1\}^{\vm \times n}$ the random matrix whose non-zero entries in row $i$ are in columns $\vec \alpha_{i,1}, \ldots, \vec\alpha_{i,k}$.

\vspace{0.3 cm}
\textbf{Upper bound:} 
Let $\vy \in \Gamma^{\vm}$ be chosen uniformly given $\vm$ and independently of $(\vec\alpha_{i,j})_{i,j}$ as above. Set 
$$\vy'= \vy \mod p,$$ 
and 
consider the linear system $\vB x = \vy'$ over $\FF_p$. Then the system $(\vB, \vy')$ is of the type studied in  \Thm~\ref{Thm_SAT}  with permutation invariant distribution $P=\delta_{(1, \ldots, 1)}$, since $\vy'$ is uniformly distributed over $\{0, \ldots, p-1\}$.
Moreover, any solution $x \in \Gamma^n$ to $\bigwedge_{i=1}^{\vm}\bc{\sum_{h=1}^k x_{\vec \alpha_{i,h}}=\vec y_i}$ yields a solution $x'  = x \mod p \in \FF_p^n$ of $\vB x' = \vy'$. 

Thus, if there is no solution to $\vB x' = \vy'$ \whp, there is also no solution to  $\bigwedge_{i=1}^{\vm}\bc{\sum_{h=1}^kx_{\vec \alpha_{i,h}}=\vec y_i}$ \whp~
 By  \Thm~\ref{Thm_SAT}, 
\begin{align*}
	\lim_{n\to\infty}\pr\brk{\exists x\in\Gamma^n:\bigwedge_{i=1}^{\vm}\bc{\sum_{h=1}^kx_{\vec \alpha_{i,h}}=\vec y_i}} = 0 \mbox{ if }d>d_k.
\end{align*}

\vspace{0.3 cm}
\textbf{Lower bound:} Here we use the following relation:
Let $R$ be any commutative ring and $A$ be a square matrix with entries in $R$. Then
\begin{align*}
A \text{ is invertible over } R \qquad \Longleftrightarrow \qquad \text{det}_R(A) \text{ is a unit in } R.
\end{align*}

Assume now that $\rk_{\FF_p}(\vB) = \vm$. Since the rank of $\vB$ is equal to the largest size of a square submatrix of $\vB$ that has non-zero determinant, we can choose such a square matrix $\vB'$ with
\begin{align*}
\text{det}_{\FF_p}(\vB') \not= 0.
\end{align*}
In particular, $\text{det}_{\ZZ}(\vB')$ is not divisible by $p$, which implies that $\text{det}_{\ZZ_{p^\ell}}(\vB')$ is a unit in $\ZZ_{p^\ell}$. By the above fact, it follows that $\vB'$ is invertible over $\ZZ_{p^\ell}$. 

Therefore, if $\rk_{\FF_p}(\vB) = \vm$, for any $y \in \Gamma^n$, the system $\vB x = y$ over $\ZZ_{p^\ell}$ has a solution and so 
\begin{align*}
	\lim_{n\to\infty}\pr\brk{\exists x\in\Gamma^n:\bigwedge_{i=1}^{\vm}\bc{\sum_{h=1}^kx_{\vec \alpha_{i,h}}=\vec y_i}} = 1 \mbox{ if }d<d_k.
\end{align*}
\end{proof}

\subsection{Proof of \Lem~\ref{Fact_Nishimori}}\label{Sec_Fact_Nishimori}
Because $\vhx\in\FF_q^n$ is uniformly distributed and independent of $\vA$, Bayes' rule shows that for any $A,x,y$,
	\begin{align*}
	\pr\brk{\vA=A,\vhy=y\mid\vm}&=\sum_{x\in\FF_q^n}\pr\brk{\vhy=y|\vA=A,\vhx=x,\vm}\pr\brk{\vA=A,\vhx=x\mid\vm}\\
		&=q^{-n}\pr\brk{\vA=A\mid\vm}\sum_{x\in\FF_q^n}\vecone\{Ax=y\}=q^{-n}\pr\brk{\vA=A\mid\vm}Z(A,y),
	\end{align*}
which gives (\ref{eqNishimori}), by noting that $\Erw[Z(\vA,\vy)\mid \vm=m]=q^{n-\vm}$.
Further, using Bayes' rule and (\ref{eqNishimori}), we see that for any $A,x,y$,
	\begin{align*}
	\pr\brk{\vhx=x|\vA=A,\vhy=y}&=\frac{\pr\brk{\vhy=y|\vA=A,\vhx=x}\pr\brk{\vA=A}\pr\brk{\vhx=x}}
		{\sum_{z\in\FF_q^n}\pr\brk{\vA=A,\vhy=y,\vhx=z}}
		=\frac{\vecone\{Ax=y\}}{Z(A,y)},
	\end{align*}
whence the first assertion follows.
Finally, the second assertion follows from (\ref{eqNishimori}) because $Z(A,y)=Z(A,y')$ for all $y,y'$ in the image of $A$.
\qed

\subsection{Proof of \Lem~\ref{Prop_kernel}}\label{Sec_kernel}
We construct the decomposition $S_0=S_0(A), \ldots, S_n=S_n(A)$ step by step. First, let 
	$$S_0= \{i \in [n]: \forall x\in\ker(A):x_i=0\}$$
be the set of coordinates set to $0$ in all vectors $x\in\ker(A)$.
Subsequently, having constructed $S_0,\ldots,S_h$ already, we either let $S_{h+1}=\emptyset$ if $S_0\cup\cdots\cup S_h=[n]$ or otherwise
pick the least index $i\in[n]\setminus(S_0\cup\cdots\cup S_h)$ and define
	\begin{align}\label{eqProp_kernel_1}
	S_{h+1}= \{j \in [n]\setminus(S_0\cup\cdots\cup S_h): \exists s \in \FF_q^*\, \forall x\in\ker(A): x_i=sx_j\}.
	\end{align}
By construction, the sets $S_0, \ldots, S_n$ are pairwise disjoint and statement (i) in \Lem~\ref{Prop_kernel} is satisfied.

To verify the other two statements we observe that
	\begin{align}\label{eqProp_kernel_2}
	\bck{\vecone\{\vx_i=\sigma\}}_{A,0}&=1/q&\mbox{for all }i\not\in S_0,\sigma\in\FF_q.
	\end{align}
Indeed, if $i\not\in S_0$ then there is $x\in\ker(A)$ such that $x_i=1$.
We extend $x$ to a basis $x,y_1,\ldots,y_l$ of $\ker(A)$.
Then the map 
	$$\varphi:\FF_q^{l+1}\to\ker(A),\qquad(\tau_0,\ldots,\tau_l)\mapsto\tau_0x+\sum_{i=1}^l\tau_iy_i$$
is an isomorphism.
Hence, the image of a uniformly random tuple $(\TAU_0,\ldots,\TAU_l)\in\FF_q^{l+1}$ is a uniformly distributed element of the kernel.
Further, given $\TAU_1,\ldots,\TAU_l$, the $i$-th coordinate of the vector $\varphi(\TAU_0,\ldots,\TAU_l)$ is uniformly distributed because
$\TAU_0$ is chosen independently of $\TAU_1,\ldots,\TAU_l$ and $x_i=1$, whence we obtain (\ref{eqProp_kernel_2}).

Combining (\ref{eqProp_kernel_1}) and (\ref{eqProp_kernel_2}) we see that statement (ii) in \Lem~\ref{Prop_kernel} is satisfied.
To prove (iii) let $i,i'$ be two indices that belong to distinct classes $S_h,S_{h'}$, $1\leq h<h'\leq n$, respectively.
Then there exist $\xi,\zeta\in\ker(A)$ such that
	\begin{align*}
	\rk\begin{pmatrix}
	\xi_i&\xi_{i'}\\
	\zeta_i&\zeta_{i'}
	\end{pmatrix}=2.
	\end{align*}
Hence, there are linear combinations $x=s\xi+t\zeta$, $x'=s'\xi+t'\zeta$, $s,s',t,t'\in\FF_q$, such that 
	\begin{align}\label{eqProp_kernel_3}
	\begin{pmatrix}
	x_i&x_{i'}\\
	x_i'&x_{i'}'
	\end{pmatrix}&=\begin{pmatrix}1&0\\0&1\end{pmatrix}.
	\end{align}	
Since $x,x'$ are linearly independent we find {$y_2,\ldots,y_{l}$} that extend $x,x'$ to a basis of $\ker(A)$, and the isomorphism
	$$\psi:\FF_q^{l+1}\to\ker(A),\qquad(\tau_0,\ldots,\tau_l)\mapsto\tau_0x+\tau_1x'+\sum_{i=2}^l\tau_iy_i$$
maps the uniform distribution on $\FF_q^{l+1}$ to the uniform distribution on $\ker(A)$.
Finally, due to (\ref{eqProp_kernel_3}) the joint distribution of the $i$-th and the $i'$-th coordinate of the vector
$\psi(\TAU_0,\ldots,\TAU_l)$ is uniform on $\FF_q\times\FF_q$, as desired.
\qed

\subsection{Proof of \Cor~\ref{Prop_hom}}\label{Sec_hom}
We begin with the following simple fact.

\begin{lemma}\label{Lemma_eqFact_simpleOverlap1} 
Let $d>0$. If \eqref{eqOverlap2} holds, then
	\begin{align}\label{eqFact_simpleOverlap1}
	\lim_{n\to\infty}\frac1{n^2}\sum_{\substack{1\leq i<j\leq n\\\sigma,\tau\in\FF_q}}\Erw\abs{\bck{\vecone\{\vx_i=\sigma,\vx_j=\tau\}}_{\vA,0}-q^{-2}}&=0.
	\end{align}
\end{lemma}
\begin{proof}
If \eqref{eqOverlap2} is satisfied, then \whp\ for all but $o(n^2)$ pairs $1\leq i<j\leq n$ there exist $h,h'\in[n]$, $h\neq h'$ such that
$i\in S_h(\vA)$ and $j\in S_{h'}(\vA)$.
For such pairs part (iii) of \Lem~\ref{Prop_kernel} shows that $\bck{\vecone\{\vx_i=\sigma,\vx_j=\tau\}}_{\vA,0}=q^{-2}$.
\end{proof}

Hence, we may assume that \eqref{eqFact_simpleOverlap1} is satisfied.
If so, then clearly there is a sequence $\eps=\eps_n=o(1)$, $\eps=\Omega((\ln n)^{-1})$, such that with high probability, 
	\begin{align}\label{eqFact_simpleOverlap1_1}
	\sum_{i,j=1}^n\abs{\bck{\vecone\{\vx_i=\sigma,\vx_j=\tau\}}_{\vA,0}-q^{-2}}&\le\eps^8 n^2&\mbox{for all }\sigma,\tau\in\FF_q.
	\end{align}
We begin by pointing out that (\ref{eqFact_simpleOverlap1_1}) implies that for any set $I\subset[n]$ and any $\sigma\in\FF_q$ the number of positions $i\in I$ such that $\vx_i=\sigma$ is concentrated about $q^{-1}|I|$.
More precisely, we have the following.

\begin{lemma}\label{Lemma_Y}
If (\ref{eqFact_simpleOverlap1_1}) holds for the matrix $A$, then for any set $I\subset[n]$ the following is true.
Let $\cY_{I,\sigma}(\vx)$ be the number of positions $i\in I$ such that $\vx_i=\sigma$.
Then $\bck{\vecone\cbc{\abs{\cY_{I,\sigma}(\vx)-q^{-1}|I|}>\eps^2n}}_{A,0}\leq\eps^2.$
\end{lemma}
\begin{proof}
Due to (\ref{eqFact_simpleOverlap1_1}) and the triangle inequality, for every $\sigma\in\FF_q$ we have
	\begin{align}\label{eqFact_simpleOverlap1_3}
	\sum_{i=1}^n\abs{\bck{\vecone\{\vx_i=\sigma\}}_{A,0}-q^{-1}}&=
		\sum_{i=1}^n\abs{\frac1{n-1}\sum_{j\neq i}\sum_{\tau\in\FF_q}\bc{\bck{\vecone\{\vx_i=\sigma,\vx_j=\tau\}}_{A,0}-q^{-2}}}
			\leq \eps^8qn.
	\end{align}
Since $\bck{\cY_{I,\sigma}(\vx)}_{A,0}=\sum_{i\in I}\bck{\vecone\{\vx_i=\sigma\}}_{A,0}$, (\ref{eqFact_simpleOverlap1_3}) yields
	\begin{align}\label{eqFact_simpleOverlap1_4}
	\abs{\bck{\cY_{I,\sigma}(\vx)}_{A,0}-q^{-1}|I|}&\leq\sum_{i=1}^n\abs{\bck{\vecone\{\vx_i=\sigma\}}_{A,0}-1/q}
		\leq\eps^8qn.
	\end{align}
Further, once more due to (\ref{eqFact_simpleOverlap1_1}) the second moment of $\cY_{I,\sigma}$ works out to be
	\begin{align}
	\bck{\cY_{I,\sigma}(\vx)^2}_{\vA,0}&=
		\sum_{i,j\in I}\bck{\vecone\{\vx_i=\sigma,\vx_j=\sigma\}}_{\vA,0}\nonumber\\
		&\leq q^{-2}|I|^2+
			\sum_{i,j\in I}\abs{\bck{\vecone\{\vx_i=\sigma,\vx_j=\sigma\}}_{\vA,0}-q^{-2}}
			\leq q^{-2}|I|^2+\eps^8n^2.\label{eqFact_simpleOverlap1_5}
	\end{align}
Since $\eps=\Omega((\ln n)^{-1})$ does not vanish too rapidly,  (\ref{eqFact_simpleOverlap1_4}), (\ref{eqFact_simpleOverlap1_5}) and Chebyshev's inequality yield
	\begin{align*}
	\bck{\vecone\cbc{\abs{\cY_{I,\sigma}(\vx)-q^{-1}|I|}>\eps^2n}}_{\vA,0}&\leq
		\eps^{-4}n^{-2}\bc{\bck{\cY_{I,\sigma}(\vx)^2}_{\vA,0}-\bck{\cY_{I,\sigma}(\vx)}_{\vA,0}^2}\leq\eps^2,
	\end{align*}
as desired.
\end{proof}

\begin{proof}[Proof of \Cor~\ref{Prop_hom}]
For $y\in\Im(\vA)$ the set of solutions to the linear system $\vA x=y$ is simply the affine translation $z+\ker(\vA)$ of the kernel
for an inverse image $z$ of $y$.
Therefore, it suffices to show that if $\vA$ satisfies (\ref{eqFact_simpleOverlap1_1}), then
	\begin{align}\label{eqFact_simpleOverlap1_2}
	\bck{\abs{\omega_{\sigma,\tau}(\vx,\vx')-q^{-2}}}_{\vA,y}
		=\bck{\abs{\omega_{\sigma,\tau}(\vx-z,\vx'-z)-q^{-2}}}_{\vA,0}
		&\leq\left(6+\frac{2}{q^2}\right)\eps^2 &\mbox{for all }\sigma,\tau\in\FF_q.
	\end{align}
We will prove that (\ref{eqFact_simpleOverlap1_2}) is true for any matrix $A$ that satisfies (\ref{eqFact_simpleOverlap1_1})
and any $y\in\Im(A)$. Our interest is in the number of entries $i \in [n]$ for which $\vx_i-z_i=\sigma$ and $\vx'_i-z_i=\tau$, and so it is natural to first look at the positions $i \in [n]$ which have $\vx_i-z_i=\sigma$ and, in a next step, for the positions \textit{among these} for which $\vx'_i-z_i=\tau$, see the following scheme:
\begin{equation*}
\setlength{\jot}{15pt}
\begin{aligned}
\vx -z &= \big(
 \quad \ast \qquad \ast \qquad \ast \qquad \smash[b]{\overbrace{\underset{\hspace{-1.2 cm} \downarrow \hspace{0.5 cm}\downarrow \hspace{0.5 cm} \downarrow}{\boxed{\begin{matrix}\sigma \quad \cdots \quad \sigma \quad \cdots \quad \sigma \\ \end{matrix}}}}^{I(\sigma,\vx-z)}} \begin{matrix}&&\ast\end{matrix}
 \hspace{0.47 cm}\big)\\
 \vx' - z&= 
  \big(
 \quad \ast \qquad \ast \qquad \ast \qquad  \boxed{\begin{matrix}\tau \quad \cdots \quad \tau\end{matrix}} \begin{matrix} & \ast && \ast &&\end{matrix} \begin{matrix}\ast &\end{matrix} \big).
 \end{aligned}
 \end{equation*}
First, as $\vx-z$ is a uniform solution of the homogeneous system, \Lem~\ref{Lemma_Y} yields
	\begin{align}\label{eqFact_simpleOverlap1_9}
	\bck{\vecone\cbc{\abs{\cY_{[n],\sigma}(\vx-z)-q^{-1}n}>\eps^2n}}_{A,0}&\leq\eps^2&\mbox{for every }\sigma\in\FF_q.
	\end{align}
Additionally, given $\vx-z$, let $I(\sigma,\vx-z)$ be the set of all $i\in [n]$ such that $\vx_i-z_i=\sigma$.
Since $\vx'$ is chosen independently of $\vx$, a second application of \Lem~\ref{Lemma_Y} shows that
	\begin{align}\label{eqFact_simpleOverlap1_10}
	\bck{\vecone\cbc{\abs{\cY_{I(\sigma,\vx-z),\tau}(\vx'-z)-q^{-1}|I(\sigma,\vx-z)|}>\eps^2n}\,\big|\,\vx-z}_{A,0}&\leq\eps^2&\mbox{for every }\tau\in\FF_q.
	\end{align}
Thus, as $\cY_{I(\sigma,\vx-z),\tau}(\vx'-z) = n \omega_{\sigma,\tau}(\vx-z,\vx'-z)$, equations (\ref{eqFact_simpleOverlap1_9}) and (\ref{eqFact_simpleOverlap1_10}) imply by averaging over $\vx-z$ that
	\begin{align*}
  \bck{\vecone\cbc{	\abs{\omega_{\sigma,\tau}(\vx-z,\vx'-z)-q^{-2}}>4\eps^2}}_{A,0}&\leq2\eps^2,
	\end{align*}
and thus
	\begin{align}\label{eqFact_simpleOverlap1_12}
	\bck{\abs{\omega_{\sigma,\tau}(\vx-z,\vx'-z)-q^{-2}}}_{A,0}&
	\leq \left(6+\frac{2}{q^2}\right)\eps^2.
	\end{align}
Finally, since \Lem~\ref{Lemma_eqFact_simpleOverlap1} implies that $\vA$ satisfies (\ref{eqFact_simpleOverlap1_1}) \whp, the assertion follows from (\ref{eqFact_simpleOverlap1_2}) and (\ref{eqFact_simpleOverlap1_12}).
\end{proof}

\subsection{Proof of \Lem~\ref{Prop_opt}}\label{Sec_Prop_opt}
The first two derivatives of $\phi$ work out to be
	\begin{align*}
	\phi'(\alpha)&=d(k-1)\alpha^{k-2}\bc{1-\alpha-\exp\bc{-d \alpha^{k-1}}},\\
	\phi''(\alpha)&=d(k-1)(k-2)\alpha^{k-3}\bc{1-\alpha-\exp\bc{-d \alpha^{k-1}}}
		-d(k-1)\alpha^{k-2}\bc{1-d(k-1)\alpha^{k-2}\exp\bc{-d \alpha^{k-1}}}.
	\end{align*}
Thus, the zeros of the first derivatives are precisely the fixed points of the function
	$$f(\alpha)=1-\exp\bc{-d \alpha^{k-1}},$$
one of which is $\alpha=0$.
Moreover, since $f(1)<0$ and thus $\phi'(1)<0$, the maximum of $\phi$ cannot be attained at $\alpha=1$.
Consequently, the maximum is attained at one of the fixed points of $f$ in $[0,1)$.
The derivatives of $f$ are
	\begin{align*}
	f'(\alpha)&=d(k-1)\alpha^{k-2}\exp\bc{-d \alpha^{k-1}},&
	f''(\alpha)&=d(k-1)\alpha^{k-3}\exp\bc{-d \alpha^{k-1}}\bc{k-2-d(k-1)\alpha^{k-1}}.
	\end{align*}
Hence, $f''$ has only one positive root, namely $\alpha=[(k-2)/(d(k-1))]^{1/(k-1)}$, and $f''(\alpha)>0$ for small $\alpha>0$.
Therefore, the set $\cA$ of fixed points of $f$ in $[0,1)$ has size $|\cA|\leq3$.
Further, if $|\cA|=3$, then the smaller positive fixed point is unstable, i.e., the derivative of $f$ is strictly greater than one at this point.

Plugging the fixed points $\alpha\in\cA$ into the second derivative of $\phi$, we find
	\begin{align*}
	\phi''(\alpha)&=-d(k-1)\alpha^{k-2}\bc{1-f'(\alpha)}.
	\end{align*}
Therefore, 
the global maximum of $\phi$ on $[0,1]$ is attained either at $\alpha=0$ or at the largest fixed point $\rho_{k,d}$ of $f$.
But while $\phi(0)=1-d/k$, the definition (\ref{eqdk}) of $d_k$ guarantees 
	\begin{align*}
	\phi(\rho_{k,d})&=\exp\bc{-d\rho_{k,d}^{k-1}}+d\rho_{k,d}^{k-1}-\frac{d(k-1)}{k}\rho_{k,d}^k-\frac dk<1-\frac dk&\mbox{for $d<d_k$.}
	\end{align*}
Hence, if $d\leq d_k$, then $\phi(\alpha)\leq1-d/k$ for all $\alpha$, 
\aco{while in the case $d>d_k$ the maximum is attained at $\rho_{k,d}$.} \qed

\section{Proof of \Prop~\ref{Prop_freeEnergy}}\label{Sec_Prop_freeEnergy}
\noindent
As outlined in \Sec~\ref{Sec_strategy}, 
we are going to couple the random linear system $\vA_n$ with $n$ variables and the system $\vA_{n+1}$ with $n+1$ variables
and calculate the expected change $\Erw[\nul(\vA_{n+1})-\nul(\vA_n)]$ in nullity.
Roughly speaking, to perform this calculation we will determine the probability that a random vector $\vx\in\ker\vA_n$, extended by a specific further entry $x_{n+1}\in\FF_q$, also belongs to the kernel of $\vA_{n+1}$.
But to this end we need to get a handle on the uniform distribution on the kernel of $\vA_n$.
While we do not know how to get a handle on this distribution directly,
in the following subsection we will see that a slight random perturbation of $\vA_n$ that does not affect the nullity  significantly very likely leaves us with an easy-to-describe distribution.
Based on this result we will subsequently carry out the coupling argument and show that the formula for the change in nullity matches the expression stated in \Prop~\ref{Prop_freeEnergy}.

\subsection{Random solutions of random 
equations}
For a matrix $A\in\FF_q^{m\times n}$ let $\mu_A\in\cP(\FF_q^n)$ signify 
the distribution 
	$$\mu_A(\sigma)=\vecone\{\sigma\in\ker(A)\}/Z(A),$$
where, recalling (\ref{eqZnullity}), $Z(A)=|\ker(A)|$.
Remembering the notion of $(\eps,2)$-symmetry from \Sec~\ref{Sec_prelims}, we write $\mu_{A,i}$ for the marginal distribution of the $i$'th entry $\vec x_i$
of a random element $\vx\in\ker(A)$ and $\mu_{A,i,j}$ for the joint distribution of the entries $(\vec x_i,\vec x_j)$.
While we may not be able to investigate the distribution $\mu_{A}$ directly, the following lemma shows that 
{`freezing'} a bounded number of randomly chosen variables to zero likely leaves us with an $(\eps,2)$-symmetric distribution. `Freezing' a bounded number of variables has the same effect as `pinning' the variables as in the {\em pinning lemma}, which were applied in other random constraint satisfaction problems, cf.\cite[Theorem 2.2]{Montanari}. However the nice properties in Lemma~\ref{Prop_kernel} due to the linear structure of the kernel of linear systems make the proof simpler than usual. Note that the lemma works for any arbitrary matrix $A$. 

\begin{lemma}\label{Lemma_0pinning} 
For any $0<\eps<1$ there exists an integer $T \in \mathbb{N}$ such that for any $A\in\FF_q^{m\times n}$ the following is true.
Choose $\vec \theta \in [T]$ uniformly at random and independent, identically distributed uniform indices $\vi_1, \vi_2, \ldots \in[n]$, which are also independent of $\vec \theta$. Let $\vI$ be {the $(\vec\theta-1) \times n$ matrix} with entries
	$\vI_{st}=\vecone\{\vi_s=t\}$.
Additionally, let
	$A[\vi_1,\ldots,\vi_{\vec \theta-1}]=\bink{A}{\vI}.$
Then $\pr\brk{\mu_{A[\vi_1,\ldots,\vi_{\vec \theta-1}]}\mbox{ is $(\eps,2)$-symmetric}}>1-\eps.$
\end{lemma}
\begin{proof}
Recall the decomposition $S_0(A), S_1(A),\ldots, S_{n}(A)$ of the index set $[n]$ corresponding to the matrix $A\in\FF_q^{m\times n}$ from \Lem~\ref{Prop_kernel}, which has the following three properties: First, all solutions $x \in \ker(A)$ have $x_i=0$ for all $i \in S_0(A)$. Second, if $i, j$ are elements of the same class $S_u(A)$ for $u \geq 1$ and $\vx$ is a uniform vector in $\ker(A)$, there exists $s \in \FF_q^{\ast}$ such that the joint distribution of $\vec x_i, \vec x_j$ is the uniform distribution on the set $\{(t, st): t \in \FF_q\}$.
Finally, if $i, j$ belong to different classes $S_u(A), S_v(A)$ with $u, v \geq 1$, then the joint distribution of $\vec x_i$ and $\vec x_j$ is the uniform distribution on $\FF_q^2$. In particular, $\vec x_i$ and $\vec x_j$ are independent. 

Hence, $\mu_{A, i,j}-\mu_{A,i}\tensor \mu_{A,j} = 0$ unless $i,j$ belong to the same class $S_u(A)$ for some $u\geq1$. The proof idea is to show that the addition of the rows $\vI$ has the impact that the sets $S_1(A[\vi_1,\ldots,\vi_{\vec \theta-1}]),\ldots, S_{n}(A[\vi_1,\ldots,\vi_{\vec \theta-1}])$ of the modified matrix are relatively small: 
	\begin{align}\label{eqLemma_0pinning1}
	\pr\brk{\max_{u\geq1}|S_u(A[\vi_1,\ldots,\vi_{\vec \theta-1}])|\geq\eps n}&<\eps.
	\end{align}
In order to prove (\ref{eqLemma_0pinning1}), it is enough to keep track of the evolution of $(S_0(A[\vi_1, \ldots, \vi_{t}]))_{t \geq 0}:$ Adding one more row to $A[\vi_1, \ldots, \vi_{t-1}]$ can only increase the set $S_0(A[\vi_1, \ldots, \vi_{t-1}])$. More precisely, if $\vi_t \in S_u(A[\vi_1, \ldots, \vi_{t-1}])$ for $u \geq 0$, then 
$$S_0(A[\vi_1, \ldots, \vi_{t}]) = S_0(A[\vi_1, \ldots, \vi_{t-1}]) \cup S_u(A[\vi_1, \ldots, \vi_{t-1}]).$$
Now set $T=\lceil 1/\eps^3 \rceil$ and for $1 \leq t \leq T$, 
\begin{align*}
\Delta_t = \left|S_0(A[\vi_1,\ldots,\vi_{t}]) \right| - \left|S_0(A[\vi_1,\ldots,\vi_{t-1}]) \right|.
\end{align*}
Assume that (\ref{eqLemma_0pinning1}) is not true, but instead we have $\pr\brk{\max_{u\geq1}|S_u(A[\vi_1,\ldots,\vi_{\vec \theta-1}])|\geq\eps n}\geq \eps.$ Because given $\vec \theta$, the index $\vec i_{\vec \theta}$ is uniformly distributed on $[n]$, this implies $\pr\bc{\Delta_{\vec \theta}\geq \eps n \vert \max_{u\geq1}|S_u(A[\vi_1,\ldots,\vi_{\vec \theta-1}])|\geq\eps n} \geq \eps$. As only the variables in the set $S_u(A[\vi_1,\ldots,\vi_{\vec \theta-1}])$, from which the index $\vec i_{\vec \theta}$ is chosen, are set to zero when going from $\vec \theta-1$ to $\vec \theta$, this implies
\begin{align} \label{Delta_inc}
\pr\bc{\Delta_{\vec \theta}\geq \eps n} = \pr\bc{\Delta_{\vec \theta}\geq \eps n\vert \max_{u\geq1}|S_u(A[\vi_1,\ldots,\vi_{\vec \theta-1}])|\geq\eps n} \cdot \pr\brk{\max_{u\geq1}|S_u(A[\vi_1,\ldots,\vi_{\vec \theta-1}])|\geq\eps n}\geq \eps^2.
\end{align}
On the other hand, as $\sum_{i=1}^T \Delta_t \leq n$ and $\Delta_t \geq 0$, there can be at most $\lfloor \frac{1}{\eps} \rfloor$ random variables among $\Delta_1, \ldots, \Delta_T$ which are greater than $\eps n$. Because $\vec \theta$ is chosen uniformly from $[T]$, the probability to observe a big increase in $S_0$ is bounded by
\begin{align*}
\pr\bc{\Delta_{\vec \theta}\geq \eps n} < \frac{1}{\eps T} \leq \eps^2,
\end{align*}
which contradicts (\ref{Delta_inc}). Thus, (\ref{eqLemma_0pinning1}) holds true.

Finally, if $\max_{u\geq1}S_u(A[\vi_1,\ldots,\vi_{\vec \theta-1}])<\eps n$, then
\begin{align*}
\sum_{j_1,j_2 \in [n]}&\TV{\mu_{A[\vi_1,\ldots,\vi_{\vec \theta-1}],j_1,j_2}-\mu_{A[\vi_1,\ldots,\vi_{\vec \theta-1}],j_1}\tensor \mu_{A[\vi_1,\ldots,\vi_{\vec \theta-1}],j_2}}\\&= \sum_{u=1}^n\sum_{j_1,j_2 \in S_u(A[\vi_1, \ldots, \vi_{\vec \theta-1}])}\TV{\mu_{A[\vi_1,\ldots,\vi_{\vec \theta-1}],j_1,j_2}-\mu_{A[\vi_1,\ldots,\vi_{\vec \theta-1}],j_1}\tensor\mu_{A[\vi_1,\ldots,\vi_{\vec \theta-1}],j_2}}
	\leq\sum_{u=1}^n|S_u(A[\vi_1,\ldots,\vi_{\vec \theta-1}])|^2<\eps n^2,
\end{align*}
and therefore $\mu_{A[\vi_1,\ldots,\vi_{\vec \theta-1}]}$ is $(\eps,2)$-symmetric.
Thus, the assertion follows from (\ref{eqLemma_0pinning1}).
\end{proof}

\noindent
Crucially, in \Lem~\ref{Lemma_0pinning} the maximal number $T-1$ of variables that we need to freeze depends only on $\eps$, but not on $m$, $n$ or $A$.

Given an $m\times n$ matrix $A$ and an integer $\theta\ge 1$, let 
\begin{equation}\label{eqAtheta}
	A[\theta]=\bink{A}{I(\theta)},
	\end{equation}
where $I(\theta)$ is the $(\theta-1)\times n$ matrix with entries $I(\theta)_{ij}=\vecone\{i=j\}$.

\begin{corollary}\label{Cor_0pinning}
For any $\eps>0$ there is an integer an integer $T \in \mathbb{N}$ such that for all sufficiently large $n$ and $\vec \theta \in [T]$ chosen uniformly and independently of $\vA_n$, 
	$\pr[\mu_{\vhA[\vec \theta]}\mbox{{ is $(\eps,2)$-symmetric}}]>1-\eps.$
\end{corollary}
\begin{proof}
This follows from \Lem~\ref{Lemma_0pinning} and the fact that the distribution of the random matrix $\vA$ is invariant under column permutations.
\end{proof}

\noindent
\Cor~\ref{Cor_0pinning} illustrates how the algebraic structure of the problem aids the proof.
To ensure $(\eps,2)$-symmetry, we just add the small matrix $I(\vec \theta)$.
In effect, as we shall see momentarily,  the coupling argument is quite natural.
By contrast, in the case of LDGM codes at positive noise a more complicated perturbation is needed to bring $(\eps,2)$-symmetry about
	(cf.~\cite[\Lem~3.5]{CKPZ}), which entails substantial technical difficulties in the coupling argument.

\subsection{The coupling argument}
The main step toward the proof of \Prop~\ref{Prop_freeEnergy} is the derivation of an upper bound on the nullity of $\vA[\vec\theta]$,
a task that we solve by way of a coupling argument.
We first establish a formula that emerges naturally from this argument.
Later we will show that this formula matches the expression from \Prop~\ref{Prop_freeEnergy}.

Let us write $\delta_u$ for the point mass at $u$. 
Moreover, let $\cQ(\FF_q)$ be the set of all probability distributions $\pi$ on $\cP(\FF_q)$ of the form
	\begin{align}\label{eqpialpha}
	\pi&=\alpha \delta_{\delta_0}+(1-\alpha)\delta_{\frac1q\sum_{\sigma\in\FF_q}\delta_\sigma}&(\alpha\in[0,1]).
	\end{align}
Thus, $\pi$ is the probability distribution on the simplex $\cP(\FF_q)$ that puts mass $\alpha$ on the atom $\delta_0\in\cP(\FF_q)$ that places mass $1$ on the single value $0$,
and mass $1-\alpha$ on the uniform distribution on $\FF_q$.
Further, define a functional
	$$\cB_{d,k}:\cQ(\FF_q)\to[0,\infty)$$
as follows.
For a given $\pi\in\cQ(\FF_q)$ let  $(\NU_{\pi,ij})_{i,j\geq1}$  be a {family of} samples from $\pi$,
let $(\va_{i})_{i\geq1}$ be vectors in $\FF_q^{*\,k}$ drawn from the distribution $P$ of the non-zero entries of the rows of $\vA$ and
let $\GAMMA$ be a $\Po(d)$-random variable, all  mutually independent.
Then with $x\perp y$ a shorthand for $\sum_{i=1}^kx_iy_i=0$ we  define
	\begin{align}\label{eqmyBFE}
	\cB_{d,k}(\pi)&=\cB^{+}_{d,k}(\pi)-d(1-1/k)\cB^{-}_{d,k}(\pi),\qquad\qquad\qquad\qquad\qquad\qquad\mbox{where}\\
	\cB^-_{d,k}(\pi)&=\Erw\brk{\ln{\sum_{\sigma\in\FF_q^k}
			\vecone\cbc{\va_1\perp\sigma}\prod_{j=1}^k\vec\nu_{\pi,1j}(\sigma_j)}},
			\qquad
	\cB^+_{d,k}(\pi)=
	\Erw\brk{\ln{\sum_{\chi\in\FF_q}\prod_{i=1}^\GAMMA\sum_{\sigma\in\FF_q^k}
		\vecone\cbc{\sigma_{k}=\chi,\,\va_i\perp\sigma}\prod_{j=1}^{k-1}\vec\nu_{\pi,ij}(\sigma_j)}}.\nonumber
	\end{align}
The definition of the domain $\cQ(\FF_q)$ guarantees that the arguments of the logarithms in the above formula are always positive.
In statistical physics jargon $\cB_{d,k}$ is called the {\em Bethe free entropy}~\cite{MM}.
Since for any matrix $A$ the number $Z(A)$ is related to the nullity through (\ref{eqZnullity}),
the following lemma bounds $\nul(\vhA[\vec\theta])$ from above, where we continue to use the notation
$\vA[\vec\theta]$ from (\ref{eqAtheta}).

\begin{lemma}\label{Prop_BFE}
For any $\xi>0$ there exists an integer $T=T(d,k,q,\xi)>0$ such that for $\vec \theta \in [T]$ chosen uniformly and independently of $\vA$,
	\begin{align}\label{eqProp_BFE_1}
	\limsup_{n\to\infty}\frac1n\Erw\brk{\ln Z(\vhA[\vec \theta])}\leq\sup_{\pi\in\cQ(\FF_q)}{\cB_{d,k}}(\pi)+\xi.
	\end{align}
\end{lemma}

We proceed to prove \Lem~\ref{Prop_BFE} via the coupling argument that was alluded to already.
To be precise, we are going to show that for large enough $T$ and $n$,
	\begin{align}\label{eqProp_BFE_2}
	\Erw\brk{\ln Z(\vhA_{n+1}[\vec \theta])}-\Erw\brk{\ln Z(\vhA_{n}[\vec \theta])}&\leq\sup_{\pi\in\cQ(\FF_q)}{\cB_{d,k}}(\pi)+\xi.
	\end{align}
Then (\ref{eqProp_BFE_1}) follows by telescoping.

To prove (\ref{eqProp_BFE_2}) we couple $\vhA_{n}[\vec \theta]$ and $\vhA_{n+1}[\vec \theta]$ as follows.
Let $\GAMMA'$ be a Poisson variable with mean
	\begin{equation}\label{eqgamma'}
	\gamma'=\frac{d(n+1)}k\cdot\bink{n}k\bink{n+1}k^{-1}=\frac{d(n+1)}k-d,
	\end{equation}
which is independent of $\vec \theta$, and let $\vA'=\vA_{n,\GAMMA'}$ be the random matrix with $n$ columns and $\GAMMA'$ rows.
Further, let $\GAMMA'',\GAMMA'''$ be Poisson variables, independent of $\GAMMA'$ and $\vec \theta$,  with means
	\begin{align}\label{eqgamma''}
	\gamma''&=\frac{dn}k-\gamma'=d\bc{1-1/k},&\gamma'''&=\frac{d(n+1)}k-\gamma'=d,
	\end{align}
respectively.
Now obtain $\vA''$ from $\vA'$ by adding $\GAMMA''$ rows, chosen mutually independently and independently of $\vA'$ from the same distribution as the rows of $\vA_n$.
Further, obtain $\vA'''$ from $\vA'$  by first adding a further column to $\vA'$ with all entries equal to $0$
and then adding $\GAMMA'''$ more rows, chosen mutually independently and independently of $\vA'$ from the same distribution as the rows of $\vA_{n+1}$ subject to the condition that the rightmost entry is non-zero.
Finally, obtain $\vA''''$ from $\vA'''$ by randomly permuting the rows.

\begin{claim}\label{Claim_coupling1}
The distribution of $\vA''$ is identical to that of $\vA_n$ and the distribution of $\vA''''$ is identical to that of $\vA_{n+1}$.
\end{claim}
\begin{proof}
The total number of rows of $\vA''$ is $\GAMMA'+\GAMMA''\sim\Po(dn/k)$, and each is chosen independently from the distribution of the rows of $\vA_n$.
Hence, $\vA''$ and $\vA_n$ are identically distributed.
Further, since the positions of the $k$ non-zero entries in each row of $\vA_{n+1}$ are chosen uniformly without replacement, the number of rows whose rightmost entry equals zero has distribution $\Po(\gamma')$.
Similarly, the number of rows whose rightmost entry is non-zero has distribution $\Po(\gamma''')$, and both these numbers are independent by Poisson thinning.
Hence, $\vA'''$ has the same distribution as $\vA_{n+1}$, except that the rows with a non-zero entry in the rightmost column stand at the bottom.
This is remedied by the random permutation, and thus $\vA''''$ and $\vA_{n+1}$ are identically distributed.
\end{proof}

\begin{claim}\label{Claim_coupling2}
The distributions of $\vA'$, $\vA_n$ have total variation distance $o(1)$.
\end{claim}
\begin{proof}
The individual rows of $\vA'$, $\vA_n$ are identically distributed, it is just that the numbers $\GAMMA',\vm$ of rows are distributed differently.
But since $\GAMMA',\vm$ have standard deviation $\Omega(\sqrt n)$ and $\gamma'- dn/k=O(1)$, the random variables $\GAMMA',\vm$ have total variation distance $o(1)$.
\end{proof}

Let $\vA'[\vec \theta],\ldots,\vA''''[\vec\theta]$ be the matrices obtained from $\vA',\ldots,\vA''''$ by adding the matrix $I(\vec\theta)$ from \Cor~\ref{Cor_0pinning} at the bottom.
Moreover, let $\vec\pi$ be the empirical distribution of the marginals 
$\mu_{\vA'[\vec\theta],1},\ldots,\mu_{\vA'[\vec\theta],n}$ of $\mu_{\vA'[\vec\theta]}$; in symbols,
	\begin{align}\label{eqpi}
	\vec\pi&=\frac1n\sum_{i=1}^n\delta_{\mu_{\vA'[\vec\theta],i}}\enspace.
	\end{align}
\Lem~\ref{Prop_kernel} shows that $\vec\pi\in\cQ(\FF_q)$.
The proofs of the following two claims constitute the main step toward the proof of \Lem~\ref{Prop_BFE}.

\begin{claim}\label{Claim_coupling3} 
For any $\xi>0$ there is $T_0>0$ such that for every $T>T_0$ for all large enough $n$ we have
	\begin{align*}
	\Erw\abs{d(1-1/k)\cB_{d,k}^-(\vec\pi)-\Erw\brk{\ln\frac{Z(\vA''[\vec\theta])}{Z(\vA'[\vec\theta])}\bigg|\vA'[\vec\theta]}}&<\xi.
	\end{align*}
\end{claim}

\begin{claim}\label{Claim_coupling4} 
For any $\xi>0$ there is $T_0>0$ such that for every $T>T_0$ for all large enough $n$ we have
	\begin{align*}
	\Erw\abs{\cB_{d,k}^+(\vec\pi)-\Erw\brk{\ln\frac{Z(\vA'''[\vec\theta])}{Z(\vA'[\vec\theta])}\bigg|\vA'[\vec
          \theta]}}&<\xi.
	\end{align*}
\end{claim}

Before we carry out the mildly technical proofs let us take a moment to explain where the formulas come from.
Concerning Claim~\ref{Claim_coupling3}, we recall that $\vA''$ is obtained from $\vA'$ by appending $\GAMMA''$
new rows $\va_1,\ldots,\va_{\GAMMA''}$.
Hence, we can interpret $Z(\vA''[\vec\theta])/Z(\vA'[\vec\theta])$ as the probability that a random vector $\vx$ in the kernel of $\vA'[\vec\theta]$ is perpendicular
to the news rows $\va_1,\ldots,\va_{\GAMMA''}$.
In formulas,
	\begin{align*}
	\frac{Z(\vA''[\vec \theta])}{Z(\vA'[\vec\theta])}&=\sum_{x\in\FF_q^n}\frac{\vecone\{x\in\ker(\vA'[\vec\theta])\}}{Z(\vA'[\vec\theta])}
		\prod_{h=1}^{\GAMMA''}\vecone\{x\perp\va_h\}=\bck{\vecone\{\forall i\leq\GAMMA'':\vx\perp\va_i\}}_{\vA'[\vec\theta],0}.
	\end{align*}
To proceed we call upon our knowledge of the measure $\mu_{\vA'[\vec\theta]}$.
Since $\vA'$ and  {$\vA_n$ are} close in total variation distance by Claim~\ref{Claim_coupling2}, \Cor~\ref{Cor_0pinning} and \Lem~\ref{lem:k-wise} show that $\mu_{\vA'[\vec\theta]}$ is very likely $(\eps,\ell)$-symmetric.
Here we can pick $\eps>0$ as small and $\ell>0$ as large as we please by making $T$ sufficiently large.
Now, crucially, the mean of $\GAMMA''$ is bounded and the positions of the non-zero entries in the $\GAMMA''$ new rows are random.
Hence, if $\GAMMA''\leq\ell/k$ is not too big, then most likely  the joint distribution of all the entries of the vector $\vx$ where one of the new rows $\va_i$ has a non-zero entry is close to a product distribution in total variation distance, thanks to $(\eps,\ell)$-symmetry.
Consequently, writing $\cJ_h$ for the set of non-zero entries of $\va_h$, with probability at least $1-\eps$,
	\begin{align}\label{eqexplainBFE}
	\ln\frac{Z(\vA''[\vec\theta])}{Z(\vA'[\vec\theta])}&=
		\sum_{h=1}^{\GAMMA''}\ln\Bigg[\sum_{\sigma\in\FF_q^{\cJ_h}}\vecone\{\sigma\perp\va_h\}
		\prod_{j\in\cJ_h}\mu_{\vA'[\vec\theta],j}(\sigma_j)\Bigg]\pm\mbox{small error term}.
	\end{align}
Finally, because for each $h$ the positions $\cJ_h$ are chosen uniformly at random, the marginals $(\mu_{\vA'[\vec\theta],j})_{j\in\cJ_h}$
are distributed as independent samples from the distribution $\vec\pi$.
Thus, up to a small error term the expectation of each summand on the right hand side of (\ref{eqexplainBFE}) coincides with $\cB_{d,k}^-(\vec\pi)$.

The intuition behind Claim~\ref{Claim_coupling4} is similar, except that now we add a column as well as a few rows.
More specifically, the additional column corresponds to a new variable $x_{n+1}$ that only appears in the new rows $\va_1,\ldots,\va_{\GAMMA'''}$.
Hence, we can interpret $Z(\vA'''[\vec\theta])/Z(\vA'[\vec\theta])$ as the expected number of ways in which a random $\vx\in\ker(\vA'[\vec\theta])$ can be enhanced
by an additional coordinate $x_{n+1}$ so that $\bink{\vx}{x_{n+1}}\perp\va_i$ for all $i=1,\ldots,\GAMMA'''$.
In symbols,
	\begin{align}\nonumber
	\frac{Z(\vA'''[\vec\theta])}{Z(\vA'[\vec\theta])}&
		=\bck{\sum_{x_{n+1}\in\FF_q}\vecone\cbc{\forall h\in[\GAMMA''']:\vec a_h\perp\bink{\vx}{x_{n+1}}}}_{\vA'[\vec\theta],0}.
	\end{align}
Similarly as before, the joint distribution of entries of $\vx$ that occur in the positions $\cJ_h$ where $\vec a_h$ is non-zero factorises with probability as close to one as we wish, provided $T$ is chosen suitably large, and then
	\begin{align}\nonumber
	\ln\frac{Z(\vA'''[\vec\theta])}{Z(\vA'[\vec\theta])}&=
			\sum_{\sigma_{n+1}\in\FF_q}
			\prod_{h\in[\GAMMA''']}\sum_{\sigma\in\FF_q^{n}}\vecone\cbc{\va_h\perp\bink{\sigma}{\sigma_{n+1}}}\prod_{j\in\cJ_h}\mu_{\vA'[\vec\theta],j}(\sigma_j)\pm\mbox{small error term}.
	\end{align}
Taking the conditional expectation of the above, we obtain the $\cB_{d,k}^+(\vec\pi)$, up to a small error.
Clearly, the only technical difficulty is to handle the various error terms, so let us deal with them now and prove Claims~\ref{Claim_coupling3} and~\ref{Claim_coupling4}.

\begin{proof}[Proof of Claim~\ref{Claim_coupling3}]
The matrix $\vA''[\vec\theta]$ is obtained from $\vA'[\vec\theta]$ by adding $\GAMMA''$ rows, whence $\rk(\vA''[\vec\theta])\leq\rk(\vA'[\vec\theta])+\GAMMA''$.
Thus, recalling that $\ln Z(A)=\nul(A)\ln q$, the basic relation $\rk(\vA''[\vec\theta])+\nul(\vA''[\vec\theta])=n=\rk(\vA'[\vec\theta])+\nul(\vA'[\vec\theta])$ yields the worst-case estimate
	\begin{equation}\label{eqClaim_coupling3_1}
	\abs{\ln\frac{Z(\vA''[\vec\theta])}{Z(\vA'[\vec\theta])}}\leq\GAMMA''\ln q.
	\end{equation}

Choose a large enough number $L=L(d,k,\xi)>0$, a small enough $\eta=\eta(L)>0$, an even smaller $\eps=\eps(\eta)>0$
and a large enough $T=T(\eps)>0$.
Let $\cF$ be the event that $\vA'[\vec\theta]$ is $(\eps,2)$-symmetric.
Then by \Cor~\ref{Cor_0pinning} and Claim~\ref{Claim_coupling2},
	\begin{align}\label{eqClaim_coupling3_pin}
	\pr\brk{\cF}&\geq1-\eps.
	\end{align}
Further, since by (\ref{eqgamma''}) the mean of $\GAMMA''$ is bounded, because the Poisson distribution has sub-exponential tails and due to the bound (\ref{eqClaim_coupling3_1}) we can choose $L$ large enough so that
	\begin{align}\label{eqClaim_coupling3_2}
	\Erw\brk{\abs{\ln\frac{Z(\vA''[\vec\theta])}{Z(\vA'[\vec\theta])}}\vecone\{\GAMMA''>L\}\,\bigg|\,\vA'[\vec\theta]}\leq\Erw\brk{\GAMMA''\vecone\{\GAMMA''>L\}}\ln q
		<\xi/4.
	\end{align}

Let $\va_1,\ldots,\va_{\GAMMA''}\in\FF_q^n$ be the rows added to $\vA'$ to obtain $\vA''$.
Moreover, for each $1\le i\le \GAMMA''$, let $\cJ_i$ be the set of all indices $j\in[n]$ such that $\va_{ij}\neq0$ and let $\cJ=\bigcup_{i\leq\GAMMA''}\cJ_i$.
In other words, $\cJ$ contains the positions of the non-zero entries in the new rows.
Since these positions are chosen uniformly, assuming that $n$ is sufficiently large we have
	\begin{align}\label{eqClaim_coupling3_3}
	\pr\brk{|\cJ|=k\GAMMA''\,\big|\,\GAMMA''\leq L}>1-\eps.
	\end{align}
Hence, most likely no variable occurs twice in the new equations.
Further, on $\cF$, \Lem~\ref{lem:k-wise} shows that $\mu_{\vA'[\vec\theta]}$ is indeed $(\eta^3,kL)$-symmetric (provided $\eps$ is sufficiently small).
Therefore, since $\cJ\subset[n]$ is a uniformly random subset given its size, we conclude that
	\begin{align}\label{eqClaim_coupling3_4}
	\pr\brk{\TV{\mu_{\vA'[\vec\theta],\cJ}-\bigotimes_{j\in\cJ}\mu_{\vA'[\vec\theta],j}}<\eta\,\bigg|\,\GAMMA''\leq L,\cF}>1-\eta.
	\end{align}
Letting $\cE=\cbc{\GAMMA''\leq L,\,|\cJ|=k\GAMMA'',\,\tv{\mu_{\vA'[\vec\theta],\cJ}-\bigotimes_{j\in\cJ}\mu_{\vA'[\vec\theta],j}}<\eta}$, we obtain from (\ref{eqClaim_coupling3_3}) and (\ref{eqClaim_coupling3_4}) that
	\begin{align}\label{eqClaim_coupling3_5}
	\pr\brk{\cE\,|\,\GAMMA''\leq L,\cF}>1-2\eta.
	\end{align}
Moreover, on the event $\cE\cap\cF$ we have
	\begin{align}
	\frac{Z(\vA''[\vec\theta])}{Z(\vA'[\vec\theta])}&
		=\bck{\vecone\{\forall h\in[\GAMMA'']:\vec a_h\perp\vx\}}_{\vA'[\vec\theta],0}
		=\prod_{h\in[\GAMMA'']}\sum_{\sigma\in\FF_q^{\cJ_h}}\vecone\{\sigma\perp\va_h\}\left(\prod_{j\in\cJ_h}\mu_{\vA'[\vec\theta],j}(\sigma_j)+O(\eta)\right)\label{eqError1}\\
		&=O(\eta)+\prod_{h\in[\GAMMA'']}\sum_{\sigma\in\FF_q^{\cJ_h}}\vecone\{\sigma\perp\va_h\}\prod_{j\in\cJ_h}\mu_{\vA'[\vec\theta],j}(\sigma_j),\label{eqClaim_coupling3_6}
	\end{align}
	where the second equation in~(\ref{eqError1}) comes from $\tv{\mu_{\vA'[\vec\theta],\cJ}-\bigotimes_{j\in\cJ}\mu_{\vA'[\vec\theta],j}}<\eta$ on $\cE$, and~(\ref{eqClaim_coupling3_6}) follows by noting that $\prod_{h\in[\GAMMA'']}\sum_{\sigma\in\FF_q^{\cJ_h}}\vecone\{\sigma\perp\va_h\}=O(1)$ on $\cE$.
We also obtain
	\begin{align}\label{eqClaim_coupling3_6a}
	q^{-k\GAMMA''}\leq\prod_{h\in[\GAMMA'']}\sum_{\sigma\in\FF_q^{\cJ_h}}\vecone\{\sigma\perp\va_h\}\prod_{j\in\cJ_h}\mu_{\vA'[\vec\theta],j}(\sigma_j)\leq1.
	\end{align}
	The upper bound 1 above is trivial, as the quantity $\prod_{h\in[\GAMMA'']}\sum_{\sigma\in\FF_q^{\cJ_h}}\vecone\{\sigma\perp\va_h\}\prod_{j\in\cJ_h}\mu_{\vA'[\vec\theta],j}(\sigma_j)$ is simply the probability that $\vec{\sigma}\perp \va_h$ simultaneously for all $h\in[\GAMMA'']$ where $\vec{\sigma}$ is a random variable with distribution given by $\bigotimes_{j\in\cJ}\mu_{\vA'[\vec\theta],j}$. For the lower bound, note that $\sigma\perp \va_h$ simultaneously for all $h\in[\GAMMA'']$ if $\sigma$ is the zero vector. The probability that $\vec{\sigma}=\vec{0}$ is at least $q^{-k\GAMMA'}$, as $\mu_{\vA'[\vec\theta],j}(0)\geq1/q$ for every $j$ by \Lem~\ref{Prop_kernel}. This yields the desired lower bound.
Therefore, choosing $\eta=\eta(L)>0$ sufficiently small and taking logs, we obtain from (\ref{eqClaim_coupling3_6}) that
	\begin{align}\label{eqClaim_coupling3_7}
	\ln\frac{Z(\vA''[\vec\theta])}{Z(\vA'[\vec\theta])}&
		\leq\xi/16+
			\sum_{h\in[\GAMMA'']}\ln\sum_{\sigma\in\FF_q^{n}}\vecone\{\sigma\perp\va_h\}\prod_{j\in\cJ_h}\mu_{\vA'[\vec\theta],j}(\sigma_j).
	\end{align}
Hence, on $\cF$, we obtain
	\begin{align}\nonumber
	\Erw\brk{\vecone\cbc\cE\cdot\ln\frac{Z(\vA''[\vec\theta])}{Z(\vA'[\vec\theta])}\bigg|\vA'[\vec
          \theta]}&\leq
		\Erw\brk{\vecone\cbc\cE\cdot\sum_{h\in[\GAMMA'']}\ln\sum_{\sigma\in\FF_q^{n}}\vecone\{\sigma\perp\va_h\}\prod_{j\in\cJ_h}\mu_{\vA'[\vec\theta],j}(\sigma_j)\bigg|\vA'[\vec \theta]}+\xi/16&[\mbox{due to (\ref{eqClaim_coupling3_7})}]\\
		&\leq\Erw\brk{\vecone\{\GAMMA''\leq L\}\cdot\sum_{h\in[\GAMMA'']}\ln\sum_{\sigma\in\FF_q^{n}}\vecone\{\sigma\perp\va_h\}\prod_{j\in\cJ_h}\mu_{\vA'[\vec\theta],j}(\sigma_j)\bigg|\vA'[\vec\theta]}+\xi/8&[\mbox{due to (\ref{eqClaim_coupling3_5}),
			\eqref{eqClaim_coupling3_6a}}].\nonumber
	\end{align}
The positions $\cJ_h$ are chosen uniformly, the distribution $P$ of the non-zero entries of the vectors $\va_h$ is permutation-invariant and
$\vec\pi$ is the empirical distribution of the marginals
$(\mu_{\vA'[\vec\theta],j})_{j\in[n]}$ by (\ref{eqpi}).
Therefore, continuing the last estimate we obtain on $\cF$,
	\begin{align}
	\Erw\brk{\vecone\cbc\cE\cdot\ln\frac{Z(\vA''[\vec\theta])}{Z(\vA'[\vec\theta])}\bigg|\vA'[\vec\theta]}&\leq
		\Erw\brk{\GAMMA''\vecone\{\GAMMA''\leq L\}}\cB_{d,k}^-(\vec\pi)+\xi/8
			\leq d(1-1/k)\cB_{d,k}^-(\vec\pi)+\xi/4.
		\label{eqClaim_coupling3_8}\end{align}
{The corresponding lower bound 
	\begin{align}\label{eqClaim_coupling3_9}
	\Erw\brk{\vecone\cbc\cE\cdot\ln\frac{Z(\vA''[\vec\theta])}{Z(\vA'[\vec\theta])}\bigg|\vA'[\vec\theta]}&
		\geq d(1-1/k)\cB_{d,k}^-(\vec\pi)-\xi/4&\mbox{on $\cF$}
	\end{align}
follows from the very same steps.}
Combining (\ref{eqClaim_coupling3_8}) and (\ref{eqClaim_coupling3_9}) we obtain
	\begin{align}\label{eqClaim_coupling3_9}
	\Erw\brk{\abs{\Erw\brk{\vecone\cbc\cE\cdot\ln\frac{Z(\vA''[\theta])}{Z(\vA'[\theta])}\bigg|\vA'[\vec
          \theta]}-d(1-1/k)\cB_{d,k}^-(\vec\pi)}}&\leq\xi/4
		&\mbox{on $\cF$}.
	\end{align}
Finally, combining (\ref{eqClaim_coupling3_1}), (\ref{eqClaim_coupling3_pin}), (\ref{eqClaim_coupling3_5}) and (\ref{eqClaim_coupling3_9}), we see that
	\begin{align*}
	\Erw&\abs{d(1-1/k)\cB_{d,k}^-(\vec\pi)-\Erw\brk{\ln\frac{Z(\vA''[\vec\theta])}{Z(\vA'[\vec\theta])}\bigg|\vA'[\vec\theta]}}\\
		&\qquad\leq
		\Erw\abs{\vecone\{\GAMMA''>L\}\ln\frac{Z(\vA''[\vec\theta])}{Z(\vA'[\vec\theta])}}+
			\Erw\brk{\vecone\{\cF\}\abs{\Erw\brk{\vecone\{\cE\}\ln\frac{Z(\vA''[\theta])}{Z(\vA'[\theta])}\bigg|\vA'[\vec\theta]}-d(1-1/k)\cB_{d,k}^-(\vec\pi)}}\\
		&\qquad\qquad	+\bc{\pr\brk{\cF^c}}+1-\pr\brk{\cE\,|\,\GAMMA''\leq L,\cF}L\ln q<\xi,
	\end{align*}
as desired.
\end{proof}

\begin{proof}[Proof of Claim~\ref{Claim_coupling4}]
As in the proof of Claim~\ref{Claim_coupling3} we begin with a crude bound on $|\ln(Z(\vA'''[\vec\theta])/Z(\vA'[\vec\theta]))|$.
Since $\vA'''[\vec\theta]$ is obtained from $\vA'[\vec\theta]$ by adding a zero column and $\GAMMA'''$ rows,
we have $\rk(\vA'''[\vec\theta])\leq\rk(\vA'[\vec\theta])+\GAMMA'''+1$.
Hence, 
	\begin{equation}\label{eqClaim_coupling4_1}
	\abs{\ln\frac{Z(\vA'''[\vec\theta])}{Z(\vA'[\vec\theta])}}\leq(1+\GAMMA''')\ln q.
	\end{equation}
Once more we choose a large $L=L(d,k,\xi)>0$, a small $\eta=\eta(L)>0$, a smaller $\eps=\eps(\eta)>0$ and a sufficiently large $T=T(\eps)>0$.
By (\ref{eqgamma''}) the mean of $\GAMMA'''$ is bounded.
Hence, due to (\ref{eqClaim_coupling4_1}) we can pick $L$ so large that
	\begin{align}\label{eqClaim_coupling4_2}
	\Erw\brk{\abs{\ln\frac{Z(\vA'''[\vec\theta])}{Z(\vA'[\vec\theta])}}\vecone\{\GAMMA'''>L\}}\leq\Erw\brk{(1+\GAMMA''')\vecone\{\GAMMA'''>L\}}\ln q
		<\xi/4.
	\end{align}
Furthermore, by \Cor~\ref{Cor_0pinning} and Claim~\ref{Claim_coupling2} the event $\cF$ that $\vA'[\vec\theta]$ is $(\eps,2)$-symmetric satisfies
	\begin{align}\label{eqClaim_coupling4_pin}
	\pr\brk{\cF}&\geq1-\eps.
	\end{align}

With {$\va_1,\ldots,\va_{\GAMMA'''}\in\FF_q^{n+1}$} the rows added to obtain $\vA'''$,
we let $\cJ_i$ be the set of all $j\in[n]$ with $\va_{ij}\neq0$.
Hence, $\cJ=\bigcup_{i\leq\GAMMA'''}\cJ_i$ contains the `old' variable indices that occur in the new equations.
For large enough $n$ we have
	\begin{align*}
	\pr\brk{|\cJ|=(k-1)\GAMMA'''\,\big|\,\GAMMA'''\leq L}>1-\eps.
	\end{align*}
Further, \Lem~\ref{lem:k-wise} shows that
	$\mu_{\vA'[\vec\theta]}\mbox{ is $(\eta^3,(k-1)L)$-symmetric}$ on $\cF$, provided $\eps,T$ were chosen suitably.
Since the $\cJ_i$ are random subsets of size $k-1$ we thus obtain
	\begin{align*}
	\pr\brk{\TV{\mu_{\vA'[\vec\theta],\cJ}-\bigotimes_{j\in\cJ}\mu_{\vA'[\vec\theta],j}}<\eta\,\bigg|\,\GAMMA'''\leq L,\cF}>1-\eta.
	\end{align*}
Consequently, the event
	$\cE=\cbc{\GAMMA'''\leq L,\,|\cJ|=(k-1)\GAMMA''',\,\tv{\mu_{\vA'[\vec\theta],\cJ}-\bigotimes_{j\in\cJ}\mu_{\vA'[\vec\theta],j}}<\eta}$
satisfies
	\begin{align}\label{eqClaim_coupling4_5}
	\pr\brk{\cE\,|\,\GAMMA'''\leq L,\cF}>1-2\eta.
	\end{align}

We can easily approximate the quotient $Z(\vA'''[\vec\theta])/Z(\vA'[\vec\theta])$ on the event $\cE\cap\cF$.
Indeed, keeping in mind that going from $\vA'[\vec\theta]$ to $\vA'''[\vec\theta]$ we picked up an additional variable $x_{n+1}$ that appears in the $\GAMMA'''$ new equations  only,
	\begin{align}\nonumber
	\frac{Z(\vA'''[\vec\theta])}{Z(\vA'[\vec\theta])}&
		=\bck{\sum_{x_{n+1}\in\FF_q}\vecone\cbc{\forall h\in[\GAMMA''']:\vec a_h\perp\bink{\vx}{x_{n+1}}}}_{\vA'[\vec\theta],0}\\
		&=O(\eta)+
			\sum_{\sigma_{n+1}\in\FF_q}
			\prod_{h\in[\GAMMA''']}\sum_{\sigma\in\FF_q^{n}}\vecone\cbc{\va_h\perp\bink{\sigma}{\sigma_{n+1}}}\prod_{j\in\cJ_h}\mu_{\vA'[\vec\theta],j}(\sigma_j).
			\label{eqClaim_coupling4_6}
	\end{align}
Further, since \Lem~\ref{Prop_kernel} shows that $\mu_{\vA'[\vec\theta],j}(0)\geq1/q$ for all $j$, we have the bounds
	\begin{align}\label{eqClaim_coupling4_6a}
	q^{-(k-1)L}\leq\sum_{\sigma_{n+1}\in\FF_q}
			\prod_{h\in[\GAMMA''']}\sum_{\sigma\in\FF_q^{n}}\vecone\cbc{\va_h\perp\bink{\sigma}{\sigma_{n+1}}}\prod_{j\in\cJ_h}\mu_{\vA'[\vec\theta],j}(\sigma_j)&\leq q.
	\end{align}
	The derivation of~(\ref{eqClaim_coupling4_6}) and~(\ref{eqClaim_coupling4_6a}) is similar to that of~(\ref{eqClaim_coupling3_6}) and~(\ref{eqClaim_coupling3_6a}).
Therefore, taking logarithms in (\ref{eqClaim_coupling4_6}) and assuming $\eta>0$ is chosen small enough, we obtain
	\begin{align}\label{eqClaim_coupling4_7}
	\ln\frac{Z(\vA'''[\vec\theta])}{Z(\vA'[\vec\theta])}&
		\leq\xi/8+
		\ln\sum_{\sigma_{n+1}\in\FF_q}
			\prod_{h\in[\GAMMA''']}\sum_{\sigma\in\FF_q^{n}}\vecone\cbc{\va_h\perp\bink{\sigma}{\sigma_{n+1}}}\prod_{j\in\cJ_h}\mu_{\vA'[\vec\theta],j}(\sigma_j).
	\end{align}
Combining (\ref{eqClaim_coupling4_5}), (\ref{eqClaim_coupling4_6a})
and (\ref{eqClaim_coupling4_7}) and recalling the definition
(\ref{eqpi}) of $\vec\pi$, we obtain that on $\cF$,
	\begin{align}\label{eqClaim_coupling4_8}
	\Erw\brk{\vecone\cbc\cE\cdot\ln\frac{Z(\vA'''[\vec\theta])}{Z(\vA'[\vec\theta])}\bigg|\vA'[\vec\theta]}&
		\leq\cB_{d,k}^+(\vec\pi)+\xi/4,\qquad\mbox{and analogously}\\
	\Erw\brk{\vecone\cbc\cE\cdot\ln\frac{Z(\vA'''[\vec\theta])}{Z(\vA'[\vec\theta])}\bigg|\vA'[\vec\theta]}&
		\geq\cB_{d,k}^+(\vec\pi)-\xi/4.
			\label{eqClaim_coupling4_9}
	\end{align}
Finally, the assertion follows from (\ref{eqClaim_coupling4_1}), (\ref{eqClaim_coupling4_pin}), (\ref{eqClaim_coupling4_5}), (\ref{eqClaim_coupling4_8}) and (\ref{eqClaim_coupling4_9}).
\end{proof}

\begin{proof}[Proof of \Lem~\ref{Prop_BFE}]
Since $\vA''''$ is obtained from $\vA'''$ merely by permuting the rows, Claim~\ref{Claim_coupling1} shows that
	\begin{equation}\label{eqProp_BFE_11}
	\Erw\brk{\ln Z(\vhA_{n+1}[\vec\theta])}-\Erw\brk{\ln Z(\vhA_{n}[\vec\theta])}=\Erw\brk{\ln\frac{Z(\vA'''[\vec\theta])}{Z(\vA''[\vec\theta])}}.
	\end{equation}
Further, Claims~\ref{Claim_coupling3} and~\ref{Claim_coupling4} and the definition (\ref{eqmyBFE}) of $\cB_{d,k}$ show that for large enough $T$ and $n$,
	\begin{align}\label{eqProp_BFE_12}
	\Erw\brk{\ln\frac{Z(\vA'''[\vec\theta])}{Z(\vA''[\vec\theta])}}&\leq\xi+\sup_{\pi\in\cQ(\FF_q)}{\cB_{d,k}(\pi)},
	\end{align}
cf.\ (\ref{eqProp_BFE_2}).
Plugging (\ref{eqProp_BFE_12}) into (\ref{eqProp_BFE_11})  and averaging on $n$ completes the proof.
\end{proof}

\subsection{The Bethe free entropy}
While the Bethe free entropy emerges naturally from the coupling argument, \Prop~\ref{Prop_freeEnergy}
claims that the nullity of $\vA$ is upper-bounded by the simpler expressions $\sup_{\alpha\in[0,1]}\phi(\alpha)$.
The following lemma shows that the two formulas are equivalent.

\begin{lemma}\label{Lemma_Bethe}
For any $\displaystyle\pi=\alpha \delta_{\delta_0}+(1-\alpha){\delta_{\frac1q\sum_{\sigma\in\FF_q}\delta_\sigma}}\in\cQ(\FF_q)$ we have
	${\cB_{d,k}(\pi)}=\phi(\alpha)\ln q.$
\end{lemma}
\begin{proof}
Let us begin with the term 
	$$\cB_{d,k}^-(\pi)=\Erw\brk{\ln{\sum_{\sigma_1,\ldots,\sigma_k\in\FF_q}
			\vecone\cbc{\sum_{j=1}^k\vec a_{1j}\sigma_j=0}\prod_{j=1}^k\vec\nu_{\pi,1j}(\sigma_j)}}.$$
The distributions $(\vec\nu_{\pi,1j})_{j=1,\ldots,k}$ are chosen from $\pi$ independently.
Hence, with probability $\alpha^k$ we have $\vec\nu_{\pi,1j}=\delta_0$ for all $j=1,\ldots,k$, and in this case 
	$$\sum_{\sigma_1,\ldots,\sigma_k\in\FF_q}
			\vecone\cbc{\sum_{j=1}^k\vec a_{1j}\sigma_j=0}\prod_{j=1}^k\vec\nu_{\pi,1j}(\sigma_j)=1$$
for any realisation of the vector {$\va_1\in  \FF_q^{*\,k}$}.
Thus, the contribution to $\cB_{d,k}^-(\pi)$ is zero.
Now assume that $\vec\nu_{\pi,1j}$ is the uniform distribution on $\FF_q$ for at least one $j$, say $j=1$.
Then given any vector $\va_1$ and any $\sigma_2,\ldots,\sigma_k\in\FF_q$ there is precisely one value $\sigma_1\in\FF_q$ for which the equation is satisfied, and 
$\vec\nu_{\pi,11}(\sigma_1)=1/q$.
Hence, 
	\begin{equation}\label{eqProp_freeEnergy_2}
	\cB_{d,k}^-(\pi)=-(1-\alpha^k)\ln q.
	\end{equation}

Moving on to
	\begin{align*}
	\cB_{d,k}(\pi)^+&=
	\Erw\brk{\ln{\sum_{\sigma_{k}\in\FF_q}\prod_{i=1}^\GAMMA\sum_{\sigma_{1},\ldots,\sigma_{k-1}\in\FF_q}
		\vecone\cbc{\sum_{j=1}^k\vec a_{ij}\sigma_j=0}\prod_{j=1}^{k-1}\vec\nu_{\pi,ij}(\sigma_j)}},
	\end{align*}
there are again two possible scenarios.
Indeed, let $\cF$ be the set of all $i\in[\GAMMA]$ such that $\NU_{\pi,ij}=\delta_0$ for all $j=1,\ldots,k-1$.
Then on the event $\{\cF=\emptyset\}$ for every $i\leq\GAMMA$ there exists $j(i)\in\{1,\ldots,k-1\}$ such that $\NU_{\pi,ij(i)}$ is the uniform distribution on $\FF_q$.
Consequently, for every given $\sigma_{k}\in\FF_q$ and for every $i$ we have
	$$\sum_{\sigma_{1},\ldots,\sigma_{k-1}\in\FF_q}
		\vecone\cbc{\sum_{j=1}^k\vec a_{ij}\sigma_j=0}\prod_{j=1}^{k-1}\vec\nu_{\pi,ij}(\sigma_j)=1/q;$$
indeed, given {$\va_i\in  \FF_q^{*\,k}$} and the values $\sigma_h$ for $h\neq j(i)$ there is precisely one `correct' $\sigma_{j(i)}$ with
$\sum_{j=1}^k\vec a_{ij}\sigma_j=0$, and this value appears with probability $1/q$.
Further, because the {$\vec \nu_{\pi,ij}$} are chosen independently, we have 
	$\pr[\cF=\emptyset|\GAMMA]=(1-\alpha^{k-1})^{\GAMMA}.$
Thus,
	\begin{align}\label{eqProp_freeEnergy_4}
	\Erw\brk{\vecone\{\cF=\emptyset\}\ln{\sum_{\sigma_{k}\in\FF_q}\prod_{i=1}^\GAMMA\sum_{\sigma_{1},\ldots,\sigma_{k-1}\in\FF_q}
		\vecone\cbc{\sum_{j=1}^k\vec a_{ij}\sigma_j=0}\prod_{j=1}^{k-1}\vec\nu_{\pi,ij}(\sigma_j)}}&=
			\Erw\brk{(1-\alpha^{k-1})^{\GAMMA}(1-\GAMMA)\ln q}.
	\end{align}
By contrast, on the event $\{\cF\neq\emptyset\}$ there is at least one $i\leq\GAMMA$ such that $\NU_{\pi,ij}=\delta_0$ for all $j$.
Hence, the only value of $\sigma_{k}$ with
$\sum_{\sigma_{1},\ldots,\sigma_{k-1}}
		\vecone\{\sum_{j=1}^k\vec a_{ij}\sigma_j=0\}\prod_{j}\vec\nu_{\pi,ij}(\sigma_j)>0$ is $\sigma_{k}=0$.
Indeed, the sum is equal to one if $\sigma_{k}=0$.
Moreover, if $i\not\in\cF$ then there is $1\le j\le k-1$ such that $\vec\nu_{\pi,ij}$ is the uniform distribution on $\FF_q$.
Since given the vector $\va_i$ and any $(\sigma_h)_{h\neq j}$ there is precisely one $\sigma_j\in\FF_q$ with $\sum_{h\in[k]}\va_{ih}\sigma_h=0$,
we find
$\sum_{\sigma_{1},\ldots,\sigma_{k-1}}
		\vecone\{\sum_{j=1}^k\vec a_{ij}\sigma_j=0\}\prod_{j}\vec\nu_{\pi,ij}(\sigma_j)=1/q$.
In summary, because in the case $\cF\neq\emptyset$ only the summand $\sigma_{k}=0$ contributes, we obtain
	\begin{align}\label{eqProp_freeEnergy_5}
	\Erw\brk{\vecone\{\cF\neq\emptyset\}\ln{\sum_{\sigma_{k}\in\FF_q}\prod_{i=1}^\GAMMA\sum_{\sigma_{1},\ldots,\sigma_{k-1}\in\FF_q}
		\vecone\cbc{\sum_{j=1}^k\vec a_{ij}\sigma_j=0}\prod_{j=1}^{k-1}\vec\nu_{\pi,ij}(\sigma_j)}}&=
			\Erw\brk{\vecone\{\cF\neq\emptyset\}(|\cF|-\GAMMA)\ln q}.
	\end{align}
Finally, due to the independence of the $\NU_{\pi,i,j}$ each $i\in[\GAMMA]$ belongs to $\cF$ with probability $\alpha^{k-1}$ independently.
Thus,
	\begin{align}\label{eqProp_freeEnergy_3}
	\Erw\brk{\vecone\{\cF\neq\emptyset\}(|\cF|-\GAMMA)}&=
		\Erw\brk{\sum_{f=1}^{\GAMMA}\bink{\GAMMA}{f}\alpha^{(k-1)f}
			(1-\alpha^{k-1})^{\GAMMA-f}(f-\GAMMA)}.
	\end{align}
Combining (\ref{eqProp_freeEnergy_4})--(\ref{eqProp_freeEnergy_3}),  we obtain
	\begin{align}\label{eqProp_freeEnergy_7}
	\frac{\cB_{d,k}(\pi)^+}{\ln q}&=\Erw\brk{(1-\alpha^{k-1})^{\GAMMA}
		+\sum_{f=0}^{\GAMMA}\bink{\GAMMA}{f}\alpha^{(k-1)f}
			(1-\alpha^{k-1})^{\GAMMA-f}(f-\GAMMA)}=\exp\bc{-d \alpha^{k-1}}+d(\alpha^{k-1}-1).
	\end{align}
Finally, the assertion follows from (\ref{eqProp_freeEnergy_2}) and (\ref{eqProp_freeEnergy_7}).
\end{proof}

\begin{proof}[Proof of \Prop~\ref{Prop_freeEnergy}]
Combining \Lem s~\ref{Prop_BFE} and~\ref{Lemma_Bethe} and using (\ref{eqZnullity}), we see that for any $\xi>0$ there is $T>0$ such that for $\vec \theta \in [T]$ chosen uniformly and independently of $\vA$,
	\begin{align}\label{eqProp_freeEnergy__1}
	\limsup_{n\to\infty}\frac1n\Erw\brk{\nul(\vhA[\vec \theta])}\leq\sup_{\alpha\in[0,1]}\phi(\alpha)+\xi.
	\end{align}
But $\vhA[\vec\theta]$ is obtained from $\vA$ by adding $\vec\theta -1$ rows.
Hence, $\rk(\vhA[\vec\theta])\leq\rk(\vA)+\vec\theta-1$ and thus $\nul(\vA)\leq\nul(\vhA[\vec \theta])+T-1$.
Therefore, (\ref{eqProp_freeEnergy__1}) implies that
	\begin{align}\label{eqProp_freeEnergy__2}
	\limsup_{n\to\infty}\frac1n\Erw\brk{\nul(\vA)}\leq
	\limsup_{n\to\infty}\frac1n\Erw\brk{\nul(\vhA[\vec\theta])}\leq\sup_{\alpha\in[0,1]}\phi(\alpha)+\xi.
	\end{align}
Taking $\xi\to0$ in (\ref{eqProp_freeEnergy__2}) completes the proof.
\end{proof}

\section{Proof of \Prop~\ref{Prop_wrap}}\label{Sec_wrap}

\noindent
We will derive \Prop~\ref{Prop_wrap} from \Prop~\ref{Prop_freeEnergy}. First we show that the upper bound provided by \Prop~\ref{Prop_freeEnergy} can hold only if the number of variables frozen to $0$ is $o(n)$ \whp\
Recalling the notation $\vA[\vec\theta]$ from (\ref{eqAtheta}), we establish the following.

\begin{lemma}\label{Lemma_unfrozen} 
If there exist $0<d<d_k$, $T>0$ such that $\limsup_{n\to\infty}\frac1n\Erw|S_0(\vA[\vec \theta])|>0$,
then there is $D<d_k$ such that 
	$$\limsup_{n\to\infty}\frac1n\Erw[\nul(\vA_{n,\vm(D,n)})]>1-D/k.$$
\end{lemma}
\begin{proof} 
Assume that there exist $0<d<d_k$, $T>0$ and $\delta>0$ such that for infinitely many $n$,
\begin{align}\label{pos_frac}
\frac1n\Erw|S_0(\vA[\vec\theta])| \geq\delta.
\end{align}
The premise (\ref{pos_frac}) then extends to all $d'>d$ by a Poisson coupling argument.
Indeed, with independent $\vec \theta$, $\vm\sim\Po(dn/k)$, $\GAMMA=\Po((d'-d)n/k)$ and $\vm'=\vm+\GAMMA\sim\Po(d'n/k)$ we can couple $\vA_{n,\vm'}[\vec\theta]$ and $\vA_{n,\vm}[\vec\theta]$ by first generating $\vA_{n,\vm}$, then adding $\GAMMA$ further independent rows to obtain $\vA_{n,\vm'}[\vec\theta]$ and finally appending $I(\vec \theta)$ at the bottom of both matrices.
In effect, $S_0(\vA_{n,\vm}[\vec \theta])\subset S_0(\vA_{n,\vm'}[\vec\theta])$.
Thus, (\ref{pos_frac}) implies that
\begin{align} \label{greater_d}
\frac1n\Erw|S_0(\vA_{n,\vm'}[\vec\theta])|\geq \delta \qquad \text{for all} \quad d' \geq d
\end{align}
for the same $\delta$ and the same subsequence of variable numbers.

With these observations in mind, we estimate $\frac 1n \frac{\partial}{\partial d}\Erw\brk{\ln Z\left(\vA_{n,\vec m(d,n) }[\vec \theta]\right)}$ for $\bar{d}<d<d_k$ under assumption (\ref{pos_frac}) and show that if (\ref{pos_frac}) were true, there would be too many solutions at parameter value $d=d_k$.
This is because if many variables are frozen to zero, then it will be more likely that a random solution to the homogeneous system extends to a solution of the homogeneous system with one more random row.

To make this precise, for each $n$, we couple $\vA_{n,\vec m}[\vec \theta]$ and $\vA_{n,\vec m + 1}[\vec\theta]$ as in the Poisson coupling argument above: first generate $\vA_{n, \vec m}$, then independently choose a sequence $1\leq\vec\alpha_{\vec m +1,1} < \cdots < \vec\alpha_{\vec m+1,k}\leq n$ of integers uniformly at random
and independently choose a $k$-tuple $\vec a= (\vec a_{\vec m + 1,1},\ldots,\vec a_{\vec m +1,k})$ from the distribution $P$.
Then obtain $\vA_{n, \vec m+1}$ by adding a row with entries $\vec a_{\vec m +1,h}$ in the positions $\vec\alpha_{\vec m+1,h}$, while all other entries are zero.
Finally, add the matrix that freezes the first $\vec \theta$ variables to zero.
Then $\ker(\vA_{n, \vec m+1}[\vec \theta])\subset\ker(\vA_{n, \vec m}[\vec \theta])$.
Further, differentiating with respect to $d$ gives
\begin{align*}
\frac 1n \frac{\partial}{\partial d}\Erw\brk{\ln Z\left(\vA_{n,\vec m }[\vec\theta]\right)} &= \frac1n\sum_{m=0}^\infty\brk{\frac{\partial}{\partial d}\pr\brk{\Po(dn/k)=m}}
		\Erw[\ln Z\left(\vA_{n,\vec m }[\vec\theta]\right)|\vm=m]\\
		&=\frac1k\sum_{m=0}^\infty\brk{\vecone\{m\geq1\}\pr\brk{\Po(dn/k)=m-1}-\pr\brk{\Po(dn/k)=m}}\Erw[\ln Z\left(\vA_{n,\vec m }[\vec\theta]\right)|\vm=m]\\
&= \frac 1k \Erw\brk{\ln \frac{Z\left(\vA_{n,\vec m + 1}[\vec\theta]\right)}{Z\left(\vA_{n,\vec m}[\vec\theta]\right)}}.
\end{align*}
To calculate the last term we consider two cases.
\begin{description}
\item[Case 1: $\vec\alpha_{\vec m + 1,1},\ldots,\vec\alpha_{\vec m +1,k} \in {S_0(\vA_{n,\vm}\brk{\vec\theta})}$]
the new equation only requires a linear combination of variables to be zero, each of which is already frozen to zero by the first $\vec m $ or the last $\vec\theta$ equations. 
In effect, the number of solutions does not change, and thus
$
Z\left(\vA_{n,\vec m + 1}[\vec\theta]\right) = Z\left(\vA_{n,\vec m}[\vec\theta]\right).
$
\item[Case 2: there is $i \in\brk k$ such that $\vec a_{\vec m + 1, i} \notin S_0(\vec A_{n,\vec m}\brk{\vec\theta})$] adding one row cannot increase the rank  by more than one, and thus $\nul(\vA_{n, \vec m}[\vec\theta])\leq\nul(\vA_{n,\vec m + 1}[\vec\theta])+1$, whence
$
Z\left(\vA_{n,\vec m + 1}[\vec\theta]\right) \geq Z\left(\vA_{n,\vec m}[\vec\theta]\right)/q.
$
\end{description}
On the subsequence under consideration, we have $\frac1n\Erw|S_0(\vA_{n, \vec m}[\vec\theta])| \geq \delta$ and thus $\frac{1}{n}|S_0(\vA_{n, \vec m}[\vec\theta])| \geq \frac{\delta}{2}$ with probability at least $\frac{\delta}{2}$. Consequently, {case 1} occurs with probability at least $(\delta/(2\eul))^k\cdot(\delta/2) >0$.
Thus, if (\ref{pos_frac}) holds, then there is $\eps=\eps(\delta)> 0$ such that
\begin{align}\label{eqgreater_d}
\frac 1n \frac{\partial}{\partial d}\Erw\brk{\ln Z\left(\vA_{n,\vec m }[\vec\theta]\right)} \geq \frac{1}{k} \left( \eps - \ln q\right).
\end{align} 
Finally, for $0<\bar{d}<d$ and any $n$,
\begin{align}\label{eqgreater_d_2}
\frac{1}{n}\Erw\brk{\ln Z\left(\vA_{n,\vm(\bar{d},n)}[\vec \theta]\right)} = \ln q \bc{1 -\Erw\brk{\rk(\vA_{n,\vm(\bar{d},n)}[\vec \theta])/n}} \geq  \ln q \bc{1 -\Erw\brk{\vm(\bar{d},n)/n}} = \ln q \bc{1-\frac{\bar{d}}{k}}
\end{align}
and thus, for $d<D<d_k$, the monotonicity property (\ref{greater_d}) together with (\ref{eqgreater_d}) and (\ref{eqgreater_d_2}) yields that
	\begin{align*}
	\limsup_{n \to \infty} \frac{1}{n}\Erw\brk{\ln Z\left(\vA_{n,\vm(D,n)}\right)}&\geq
	\limsup_{n \to \infty} \frac{1}{n}\Erw\brk{\ln Z\left(\vA_{n,\vm(D,n)}[\vec \theta]\right)}\\
	&> \limsup_{n \to \infty} \frac{1}{n}\Erw\brk{\ln Z\left(\vA_{n,\vm(d,n)}[\vec \theta]\right)} + \limsup_{n \to \infty} \int_{d}^D \frac{1}{n}  \frac{\partial}{\partial \bar{d}}\Erw\brk{\ln Z\left(\vA_{n,\vec m\bc{\bar{d},n} }[\vec\theta]\right)}\text{d}\bar{d} \\
	& > \left(1 - D/k \right)\ln q,
	\end{align*}
whence the assertion follows from (\ref{eqZnullity}). 
\end{proof}

\begin{lemma}\label{Lemma_freeze}
For any $\eps>0$ there exist $T>0$ and $n_0>0$ such that for all $n>n_0$, all $d>0$ and uniform $\vec \theta \in [T]$, the following is true.
If $\pr\brk{\max_{i\geq0}|S_i(\vA)|\geq\eps n}>\eps$, then $\pr\brk{S_0(\vA[\vec\theta])\geq\eps n}>\eps/2.$
\end{lemma}
\begin{proof}
Fix	$\eps>0$ and consider the decomposition $S_0(\vA),\ldots,S_n(\vA)$ promised by \Lem~\ref{Prop_kernel}.
As 
$$\pr\brk{S_0(\vA[\vec\theta])\geq\eps n}\geq \pr\brk{S_0(\vA[\vec\theta])\geq\eps n\,\bigg|\,\max_{i\geq0}|S_i(\vA)|\geq\eps n}\cdot \pr\brk{\max_{i\geq0}|S_i(\vA)|\geq\eps n},$$
it suffices to prove that
\begin{align}\label{eqLemma_freeze1}
\pr\brk{S_0(\vA[\vec\theta])\geq\eps n\,\bigg|\,\max_{i\geq0}|S_i(\vA)|\geq\eps n}\geq\frac12.
\end{align}
Denote by $\vA[t]$ the matrix $\vA$ with the deterministic-size $t
\times n$ matrix $I(t), I_{i,j}(t)=\vecone\cbc{i=j},$ appended at the
bottom. As argued in the proof of \Lem~\ref{Lemma_0pinning}, the
addition of the equations $x_{t}=0$, $t\in 0, \ldots,\vec \theta-1$, does not only freeze index $t$ to zero, but also all indices in the class $S_u(\vA[t-1])$ that index $t$ belongs to: $S_u(\vA[t-1])\subset S_0(\vA[t])$. Now, the partition $S_0(\vA[t]), \ldots, S_n(\vA[t])$ changes when $t$ increases, but sets can only merge during this process: Hence, if $t \notin S_u(\vA[t-1])$, the elements in $S_u(\vA[t-1])$ will belong to a set $S_{u'}(\vA[t])$ in the next step which has $|S_{u'}(\vA[t])|\geq |S_{u}(\vA[t])|$.

Assume that there is $u \in [n]$ with $|S_u(\vA)|\geq\eps n$. The probability of the event $\cbc{S_0(\vA[\vec\theta])<\eps n}$ is upper bounded by the probability that no index $1, \ldots, \vec \theta-1$ belongs to $S_u(\vA)$. This in turn is upper bounded by the probability that a geometric random variable with success parameter $\eps$ is greater than $\vec \theta-1$.
Hence,
	\begin{align*}
	\pr\brk{S_0(\vA[\vec \theta])<\eps n\,\bigg|\,\max_{i\geq0}|S_i(\vA)|\geq\eps n} = \sum_{t=1}^T\frac{1}{T} \pr\brk{S_0(\vA[\vi_1,\ldots,\vi_{t-1}])< \eps n, \,\bigg|\,\max_{i\geq0}|S_i(\vA)|\geq\eps n} \leq \frac{1-\bc{1-\eps}^{T}}{T\eps}.
	\end{align*}
By choosing $T>0$ large enough, we obtain (\ref{eqLemma_freeze1}).
\end{proof}

\begin{proof}[Proof of  \Prop~\ref{Prop_wrap}] Let $d<d_k$. For any $\eps>0$ and $n$ we have the upper bound
\begin{align}\label{eqProp_wrap}
\frac1n\Erw[\max_{i\geq0}|S_i(\vA)|] \leq \ \pr\brk{\max_{i\geq0}|S_i(\vA)|] \geq \eps n}+ \eps.
\end{align}
Hence, if $\limsup_{n\to\infty}\frac1n\Erw[\max_{i\geq0}|S_i(\vA)|]>0$, then there exists $\eps>0$ such that
$\pr\brk{\max_{i\geq0}|S_i(\vA)|] \geq \eps n} > \eps$ for infinitely many $n$.
Thus, by Lemma \ref{Lemma_freeze} there is a $T>0$ such that for $\vec \theta \in [T]$ chosen uniformly, $\pr\brk{S_0(\vA[\vec\theta])\geq\eps n}>\eps/2$ for infinitely many $n$.
Therefore, the assertion follows from Lemma \ref{Lemma_unfrozen}. 
\end{proof}

\section{Proof of \Prop~\ref{Prop_smm}}\label{Sec_Prop_smm}

\noindent
Starting from the overlap condition \eqref{eqOverlap1}, we will set up a truncated second moment calculation where we only need to consider pairs of solutions with overlap close to $\bar\omega$.
To this end, for $A\in\FF_q^{m\times n}$, $y\in\FF_q^m$ and $\eps>0$ we define
	\begin{align}\label{eqOverlap11}
	\cZ_\eps(A,y)&=\vecone\cbc{\bck{\tv{\omega(\vec x,\vec x')-\bar\omega}}_{A,y}<\eps}\cdot Z(A,y).
	\end{align}
Thus, we discard linear systems whose typical overlaps $\omega(\vec x,\vec x')$ fail to concentrate about $\bar\omega$.
To prove \Prop~\ref{Prop_smm} we calculate the first and the second moment of the random variable $\cZ_\eps$ from \eqref{eqOverlap11}.
The following lemma deals with the first moment; the key proof ingredient is the Nishimori identity (\ref{eqNishimori}).

\begin{lemma}\label{Lemma_EZ}
Assume that (\ref{eqOverlap1}) is satisfied for $d>0$.
Then for any $\eps>0$ we have
	\begin{align}\label{eqLemma_EZ}
	\lim_{n\to\infty}\pr\brk{\Erw[\cZ_\eps(\vA,\vy)|\vm]\geq (1-\eps)q^{n-\vm}}&=1.
	\end{align}
	\end{lemma}
\begin{proof}
Writing the expectation out as a sum over all possible $m\times n$-matrices $A$ and all $y\in\FF_q^m$ and applying (\ref{eqNishimori}), we obtain
	\begin{align*}
	\Erw[\cZ_\eps(\vA,\vy)|\vm=m]&=\Erw\brk{\vecone\cbc{\bck{\tv{\omega(\vec x,\vec x')-\bar\omega}}_{\vA,\vy}<\eps}Z(\vA,\vy)\bigg|\vm=m}\\
		&=q^{n-m}\sum_{A,y}
			\vecone\cbc{\bck{\tv{\omega(\vec x,\vec x')-\bar\omega}}_{A,y}<\eps}Z(A,y)q^{-n}\pr\brk{\vA=A|\vm=m}\\
		&=q^{n-m}\pr\brk{\bck{\tv{\omega(\vec x,\vec x')-\bar\omega}}_{\vA,\vhy}<\eps\big|\vm=m}.
	\end{align*}
Thus, the assertion follows from \eqref{eqOverlap1}. 
\end{proof}

\noindent
Comparing (\ref{eqLemma_EZ}) with the vanilla first moment (\ref{eqvanilla1st}), we see that the truncation (\ref{eqOverlap11}) does not diminish the first moment.
The following lemma summarises the analysis of the second moment of the truncated variable.

\begin{lemma}\label{Lemma_EZ2}
Assume that (\ref{eqOverlap1}) is satisfied for $d>0$.
Then there exists $\eps_0=\eps(d,q,k)>0$ such that for every $0<\eps<\eps_0$ for large enough $n$ and all $m\leq 2dn/k$, 
	\begin{align}\label{eqLemma_EZ2}
	\lim_{n\to\infty}\pr\brk{\Erw[\cZ_\eps(\vA,\vy)^2|\vm=m]\leq (1+\eps^{1/4})q^{2(n-m)}}&=1.
	\end{align}
\end{lemma}

\noindent
The proof of \Lem~\ref{Lemma_EZ2} is based on the Laplace method.
To apply it, we bound $\cZ_\eps(\vA,\vy)^2$ {from above} by a further random variable whose mean we can calculate explicitly.
Specifically, let
	\begin{align}\label{eqfZ}
	\fZ_\eps(\vA,\vy)&=\sum_{x,x'\in\FF_q^n}\vecone\{\tv{\omega(x,x')-\bar\omega}\leq\eps,\,\vA x=\vA x'=\vy\}.
	\end{align}
Thus, $\fZ_\eps(\vA,\vy)$ is the number of pairs $x,x'$ of solutions whose overlap is within $\eps$ of $\bar\omega$ in total variation distance.

\begin{claim}\label{Claim_EZ2_1}
For any $0<\eps<1$ we have $\cZ_\eps(\vA,\vy)^2\leq\fZ_{\sqrt\eps}(\vA,\vy)/(1-\sqrt\eps)$.
\end{claim}
\begin{proof}
If $\bck{\tv{\omega(\vec x,\vec x')-\bar\omega}}_{\vA,\vy}<\eps$, then clearly
$\bck{\vecone\{\tv{\omega(\vec x,\vec x')-\bar\omega}\leq\sqrt\eps\}}_{\vA,\vy}\geq1-\sqrt\eps$.
Hence,
	\begin{align*}
	\fZ_{\sqrt\eps}(\vA,\vy)&=\bck{\vecone\{\tv{\omega(\vec x,\vec x')-\bar\omega}\leq\sqrt\eps\}}_{\vA,\vy}Z(\vA,\vy)^2
		\geq(1-\sqrt\eps)Z(\vA,\vy)^2\geq(1-\sqrt\eps)\cZ_\eps(\vA,\vy)^2.
	\end{align*}
On the other hand, if $\bck{\tv{\omega(\vec x,\vec x')-\bar\omega}}_{\vA,\vy}\geq\eps$, then $0=\cZ_\eps(\vA,\vy)^2\leq\fZ_{\sqrt\eps}(\vA,\vy)$.
\end{proof}
	
\noindent
We are going to show that there is a constant $C>0$ such that for small enough $\eps>0$,
\begin{align}\label{eqLemma_EZ2_1}
\Erw[\fZ_\eps(\vA,\vy)|\vm]&\leq q^{2(n-\vm)+C\eps}.
\end{align}
Then (\ref{eqLemma_EZ2}) is immediate from Claim~\ref{Claim_EZ2_1}.
To establish (\ref{eqLemma_EZ2_1})  we derive an explicit formula as an upper bound for $\Erw[\fZ_\eps(\vA,\vy)|\vm]$.
Let $\cO_\eps(n)$ be the set of all probability distributions $\omega\in\cP(\FF_q\times\FF_q)$ such that
$\tv{\omega-\bar\omega}\leq\eps$ and $n\omega(\sigma,\tau)$ is an integer for all $\sigma,\tau\in\FF_q$.
Thus, $\cO_\eps(n)$ contains all possible overlaps that may contribute to the sum on the right hand side of (\ref{eqfZ}).
Let $H(\omega)=-\sum_{\sigma,\tau\in\FF_q}\omega_{\sigma,\tau}\ln\omega_{\sigma,\tau}$ denote the entropy of $\omega\in\cP(\FF_q\times\FF_q)$.
Additionally,  we recall that $P\in\cP(\FF_q^{*\,k})$ signifies the distribution of the non-zero entries of the rows of $\vA$.

\begin{claim}\label{Claim_formula}
	For $\omega\in\cP(\FF_q\times\FF_q)$ let
	\begin{align*}
	\Phi(\omega)&=\ln\brk{\sum_{\substack{a_1,\ldots,a_k,y\in\FF_q}}
		\sum_{\sigma,\tau\in\FF_q^k}
		\vecone\cbc{\sum_{i=1}^ka_i\sigma_i=\sum_{i=1}^ka_i\tau_i=y}
		\frac{P(a_1,\ldots,a_k)}q\prod_{i=1}^k\omega_{\sigma_i,\tau_i}}.
	\end{align*}
There exist $C=C(d,k,q)>0$ and $\eps_0=\eps_0(d,k,q)>0$ such that for all $0<\eps<\eps_0(d,k,q)$ 
and all $m\leq 2dn/k$ 
we have 
	\begin{align}\label{eqClaim_formula}
	\Erw[\fZ_\eps(\vA,\vy)|\vm=m]&\leq(2\pi n)^{(1-q^2)/2}\sum_{\omega\in\cO_\eps(n)}
	\frac{\exp(n H(\omega)+m\Phi(\omega)+C\eps)}{\prod_{\sigma,\tau\in\FF_q}\sqrt{\omega_{\sigma,\tau}}}.
	\end{align}
\end{claim}
\begin{proof}
	For a given $\omega$ the number of $x,x'$ such that $\omega(x,x')=\omega$ is given by the multinomial coefficient
${n\choose{n\omega}}.$
Stirling's formula shows that, uniformly for all $\omega\in\cO_\eps(n)$,
	\begin{align}
	{n\choose{n\omega}}&\leq(2\pi n)^{{(1-q^2)}/{2}}\exp\left\{O(1/n)-n\sum_{\sigma,\tau\in\FF_q}\omega_{\sigma,\tau}\ln\omega_{\sigma,\tau}\right\}\prod_{\sigma,\tau\in\FF_q}\omega_{\sigma,\tau}^{-1/2}
		=(2\pi n)^{{(1-q^2)}/{2}}\frac{\exp\left\{nH(\omega)+O(1/n)\right\}}{\prod_{\sigma,\tau\in\FF_q}\sqrt{\omega_{\sigma,\tau}}}.\label{eqClaim_formula_1}
	\end{align}

Next, fix the overlap $\omega$ and $x,x'$ with overlap $\omega$.
Because the distribution $P$ is permutation invariant and the positions of the non-zero entries are chosen without replacement, the probability that they satisfy any one of the $\vm$ equations of the random linear system equals
	\begin{align}
	S(\omega)&
= \sum_{\substack{a_1,\ldots,a_k,y\in\FF_q}}\sum_{\substack{ 1\leq i_1 < \ldots < i_k \leq n}} \frac{P(a_1,\ldots,a_k)}{q n (n-1) \cdots (n-k+1)}\vecone\cbc{\sum_{\ell=1}^ka_{\ell}x_{i_{\ell}}=\sum_{\ell=1}^ka_{\ell}x'_{i_{\ell}}=y} \nonumber\\
&=		\sum_{\substack{a_1,\ldots,a_k,y}}\sum_{\sigma,\tau\in\FF_q^k} \prod_{i=1}^k\bc{\omega_{\sigma_i,\tau_i}n-\sum_{h<i}\vecone\{(\sigma_h,\tau_h)=(\sigma_i,\tau_i)\}}\frac{P(a_1,\ldots,a_k)}{q n (n-1) \cdots (n-k+1)}
		\vecone\cbc{\sum_{\ell=1}^ka_{\ell}\sigma_{\ell}=\sum_{i=1}^ka_{\ell}\tau_{\ell}=y}
		\nonumber\\
& = \left( 1 + \frac{k(k-1)}{2n}\right)\exp(\Phi(\omega)) -T(\omega)+ O\left(n^{-2}\right),\hspace{6cm}\mbox{where}
	\label{eqClaim_formula_2}\\
T(\omega)&= \sum_{\substack{a_1,\ldots,a_k,y}}\frac{P(a_1,\ldots,a_k)}{qn} \sum_{1\leq h<j\leq k}\sum_{\sigma,\tau\in\FF_q^k}
	\vecone\{(\sigma_h,\tau_h)=(\sigma_j, \tau_j)\}
	\vecone\cbc{\sum_{\ell=1}^ka_{\ell}\sigma_{\ell}=\sum_{i=1}^ka_{\ell}\tau_{\ell}=y}\prod_{\substack{i=1\\ i\not=j}}^k\omega_{\sigma_i,\tau_i}.
		\nonumber
	\end{align}
For the second equation in~(\ref{eqClaim_formula_2}) above, the summation over all choices of $i_1,\ldots, i_k$ is expressed by first fixing the values of $x$ and $x'$ on these coordinates, denoted by $\sigma$ and $\tau$, and then choosing indices $i_1,\ldots, i_k$ matching $\sigma$ and $\tau$. Thus, $\sigma$ and $\tau$ range over $\FF_q^k$, and given $(\sigma, \tau)$, $ \prod_{i=1}^k\bc{\omega_{\sigma_i,\tau_i}n-\sum_{h<i}\vecone\{(\sigma_h,\tau_h)=(\sigma_i,\tau_i)\}}$ is the number of choices for $i_1,\ldots, i_k$ where $x_{i_1,\ldots, i_k}=\sigma$ and $x'_{i_1,\ldots, i_k}=\tau$, given that the overlap of $x$ and $x'$ is $\omega$. The third equation in~(\ref{eqClaim_formula_2}) is obtained by noting that
\begin{align*}
 \prod_{i=1}^k\bc{\omega_{\sigma_i,\tau_i}n-\sum_{h<i}\vecone\{(\sigma_h,\tau_h)=(\sigma_i,\tau_i)\}} &= n^k \prod_{i=1}^k \omega_{\sigma_i,\tau_i} +\sum_{1\leq h<j\leq k}
	\vecone\{(\sigma_h,\tau_h)=(\sigma_j, \tau_j)\}\prod_{\substack{i=1\\ i\not=j}}^k\omega_{\sigma_i,\tau_i}+O(n^{-2})\\
	\frac{1}{n (n-1) \cdots (n-k+1)} &= n^{-k} \left(1+\frac{k(k-1)}{2n}+O(n^{-2})\right). 
\end{align*}

We interpret $T(\omega)$ probabilistically.
Namely, let $\vec a\in\FF_q^{*\,k}$ be a vector drawn from the distribution $P$ and let $\vec\xi\in\FF_q$ be uniformly distributed and independent of $\va$.
Further, let $(\vec\sigma_1,\vec\tau_1), \ldots, (\vec\sigma_k,\vec\tau_k)\in\FF_q\times\FF_q$ be drawn from $\omega$, mutually independently and independently of $\va$ and $\vec\xi$.
Then
\begin{align*}
T(\omega)&= \frac1n \sum_{1\leq h<j\leq k}\pr\brk{(\va_h+\va_j)\begin{pmatrix}\vec\sigma_h \\ \vec\tau_h \end{pmatrix}+ \sum_{\substack{\ell=1 \\ \ell \not=h,j}}^k \vec a_{\ell}\begin{pmatrix}\vec \sigma_{\ell} \\ \vec\tau_{\ell} \end{pmatrix} = \begin{pmatrix} \vec \xi \\ \vec \xi \end{pmatrix}}. 
\end{align*}
From this equation it is immediate that $T(\bar\omega)=\frac1n \bink k2q^{-2}=\frac1n\bink k2\exp(\Phi(\bar\omega))$;
indeed, if we fix $\ell\neq h,j$, then given any $\vec a,\vec\xi$ and any $(\SIGMA_i,\TAU_i)$, $i\neq h,j,\ell$, there is precisely one
pair $(\SIGMA_\ell,\TAU_\ell)$ for which the linear equation is satisfied, and this pair comes up with probability $q^{-2}$.
Since both $T(\omega)$ and $\exp(\Phi(\omega))$ have continuous derivatives, there exists $C'=C'(q,k)>0$ such that 
$|T(\omega)-\frac1n\bink k2\exp(\Phi(\omega))|\leq C'\eps$ for all $\omega\in\cO_\eps(n)$.
Therefore, (\ref{eqClaim_formula_2}) implies that
	\begin{align}\label{eqClaim_formula_3}
	\pr\brk{\vA x=\vA x'=\vy\mid\vm}&\leq\exp\bc{\vm\Phi(\omega)+C\eps}
	\end{align}
for a certain $C=C(d,q,k)>0$.
Combining (\ref{eqClaim_formula_1}) and (\ref{eqClaim_formula_3}) completes the proof.
\end{proof}

\begin{claim}\label{Claim_derivative1}
	The Jacobi matrix of $\Phi$ at the point $\bar\omega$ satisfies $D\Phi(\bar\omega)\zeta=0$ for every vector $\zeta\perp\vecone$.
\end{claim}
\begin{proof}
Let $\vec a\in\FF_q^{*\,k}$ be a vector drawn from the distribution $P$ and let $\vec\xi\in\FF_q$ be uniformly distributed and independent of $\va$.
Also let $(\vec\sigma_1,\vec\tau_1), \ldots, (\vec\sigma_k,\vec\tau_k)\in\FF_q\times\FF_q$ be drawn from $\omega$, mutually independently and independently of $\va$ and $\vec\xi$.
Then for $s,t\in\FF_q$ we have
 \begin{align}\nonumber
	 \frac{\partial \Phi(\omega)}{\partial \omega_{s,t}}&=\frac{1}{\exp\{\Phi(\omega)\}}\brk{\sum_{a_1,\ldots,a_k,y\in\FF_q,\sigma,\tau\in\FF_q^k}
	 	\vecone\cbc{\sum_{i=1}^ka_i\sigma_i=\sum_{i=1}^ka_i\tau_i=y}
	 	\frac{P(a_1,\ldots,a_k)}q\sum_{j=1}^k\vecone\cbc{{\sigma_j\choose \tau_j}={s\choose t}}\prod_{i\ne j}\omega_{\sigma_i,\tau_i}}\\
	 	&=\frac{1}{\exp\{\Phi(\omega)\}}\sum_{j=1}^k\pr\brk{{\va_j\begin{pmatrix}s \\ t \end{pmatrix}+ \sum_{\substack{\ell=1 \\ \ell \not=j}}^k \vec a_{\ell}\begin{pmatrix}\vec \sigma_{\ell} \\ \vec\tau_{\ell} \end{pmatrix} = \begin{pmatrix} \vec\xi \\ \vec\xi \end{pmatrix}}} .
			\label{eqClaim_derivative1_1}
 \end{align}
At the point $\omega=\bar\omega$ each of these summands equals $q^{-2}$ because if we fix $\ell\neq j$, then
for any given $\va$, $\vec\xi$ and any given $(\SIGMA_i,\TAU_i)_{i\neq\ell}$ precisely one pair $(\SIGMA_\ell,\TAU_\ell)$ will satisfy the equation, and this pair is chosen with probability $q^{-2}$.
Hence, all entries of $D\Phi(\bar\omega)$ are equal to $k$.
\end{proof}

\begin{claim}\label{Claim_derivative2}
	The Hessian of $\Phi$ at the point $\bar\omega$ satisfies $\scal{D^2\Phi(\bar\omega)\zeta}\zeta=0$ for every vector $\zeta\perp\vecone$.
\end{claim}
\begin{proof}
Fix $s,s',t,t'\in[q]$ and define $\varphi(\omega)=\exp\{\Phi(\omega)\}$. Then
	\begin{align}\label{eqClaim_derivative2_1}
	{ \frac{\partial^2 \Phi(\omega)}{\partial \omega_{s,t}\partial \omega_{s',t'}}
	 =\frac{ 1}{\varphi(\omega)} \frac{\partial^2 \varphi(\omega)}{\partial \omega_{s,t}\partial \omega_{s',t'}}- \frac{ 1}{\varphi(\omega)^2}\frac{\partial \varphi(\omega)}{\partial \omega_{s,t}}\cdot  \frac{\partial \varphi(\omega)}{\partial \omega_{s',t'}}.}
	\end{align}
The second partial derivatives work out to be
	\begin{align*}
	 \frac{\partial^2 \varphi(\omega)}{\partial \omega_{s,t}\partial \omega_{s',t'}}=\sum_{\substack{a_1,\ldots,a_k,y\in\FF_q\\\sigma,\tau\in\FF_q^k}}
	 	\vecone\cbc{\sum_{i=1}^ka_i\sigma_i=\sum_{i=1}^ka_i\tau_i=y}
	 	\frac{P(a_1,\ldots,a_k)}q\sum_{j\ne j'}\vecone\cbc{{\sigma_j\choose \tau_j}={s\choose t},{\sigma_{j'} \choose \tau_{j'}}={s'\choose t'}}\prod_{i\ne j,j'}\omega_{\sigma_i,\tau_i}.
	\end{align*}
As in the case of the Jacobi matrix, at $\omega=\bar\omega$ we obtain
\begin{align}	\label{eqClaim_derivative2_2}
 \frac{\partial^2 \varphi(\bar\omega)}{\partial \omega_{s,t}\partial \omega_{s',t'}}=	\sum_{1\leq j<h\leq k}\pr\brk{{\va_j \begin{pmatrix} s \\ t \end{pmatrix}+ \va_h \begin{pmatrix} s'\\ t' \end{pmatrix}+ \sum_{\substack{\ell=1 \\ \ell \not=j,h}}^k \vec a_{\ell}\begin{pmatrix}\vec \sigma_{\ell} \\ \vec\tau_{\ell} \end{pmatrix} = {\begin{pmatrix} \vec \xi \\ \vec \xi \end{pmatrix}}}}  = \frac{k(k-1)}{q^2}.
 \end{align}
The assertion follows from (\ref{eqClaim_derivative2_1}), (\ref{eqClaim_derivative2_2}) and Claim~\ref{Claim_derivative1}.
\end{proof}

\begin{proof}[Proof of \Lem~\ref{Lemma_EZ2}]
This is now a standard application of the Laplace method.
Since
	\begin{align*}
	\frac{\partial H}{\partial \omega_{s,t}}&=-(1+\ln \omega_{s,t}),&
		\frac{\partial^2 H}{\partial \omega_{s,t}\partial\omega_{\sigma,\tau}}&=
				-\frac{\vecone\{(s,t)=(\sigma,\tau)\}}{\omega_{s,t}}\qquad\mbox{for all }s,t,\sigma,\tau\in\FF_q,
		\end{align*}
the Jacobi and the Hesse matrix of the entropy {function} satisfy
	$DH(\bar\omega)\zeta=0$, $\scal{D^2H(\bar\omega)\zeta}\zeta=-q^2\|\zeta\|_2^2$ for all $\zeta\perp\vecone$.
Since the third derivatives of $H(\omega)$, $\Phi(\omega)$  are uniformly bounded near $\bar\omega$,  
Claims~\ref{Claim_derivative1} and~\ref{Claim_derivative2} and Taylor's formula yield for small $\eps>0$,
		\begin{align*}
			nH(\omega)+\vm\Phi(\omega) &= nH(\bar{\omega})+\vm\Phi(\bar{\omega})-\frac12nq^2\|\omega-\bar\omega\|_2^2+ 
				nO\left({\|\omega-\bar\omega\|_2^3}\right)&\mbox{uniformly for all }\omega\in\cO_\eps(n).
		\end{align*}
Further, by Claim \ref{Claim_formula} there exist $\eps_0=\eps_0(d,k,q)>0$ and $C=C(d,k,q)>0$ such that for all $0<\eps<\eps_0(d,k,q)$ we have 
		\begin{align}\label{eqLaplace1}
		\Erw[\fZ_\eps(\vA,\vy)|\vm]&\leq(2\pi n)^{\frac{1-q^2}2}\exp\{n H(\bar{\omega})+\vm\Phi(\bar{\omega})+C\eps\}\cdot S,&\mbox{where}\\
		S&=\sum_{\omega\in\cO_\eps(n)}
			\exp\left\{-\frac{nq^2}{2}\sum_{s,t\in\FF_q}\left(\omega_{s,t}-q^{-2}\right)^2 + {nO\left(\|\omega-q^{-2}\vecone\|_2^3\right)}\right\}
				\prod_{s,t\in\FF_q}\omega_{s,t}^{-1/2}.\nonumber
		\end{align}
Approximating the sum by an integral, introducing $\vec z=(z_{\sigma,\tau})_{(\sigma,\tau)\in\FF_q^2\setminus\{(0,0)\}}$ and setting $z_{0,0}=1-\sum_{(\sigma,\tau)\neq(0,0)}z_{\sigma,\tau}$, we obtain
	\begin{align}\label{eqLaplace2}
	S&=(1+o(1))q^{q^2}n^{q^2-1}\int_{\mathbb{R}^{q^2-1}}\exp\left\{-\frac{nq^2}{2}\sum_{s,t\in\FF_q}\left(z_{s,t}-q^{-2}\right)^2\right\}\dd\vec z
		=(1+o(1))qn^{\frac{q^2-1}2}\int_{\mathbb{R}^{q^2-1}}\exp\left\{-\frac{1}{2}\sum_{s,t\in\FF_q}z_{s,t}^2\right\}\dd\vec z,
	\end{align}
	where the error $nO(\|\omega-q^{-2}\vecone\|_2^3)$ below~(\ref{eqLaplace1}) is absorbed by $o(1)$ in the integration, $q^{q^2}$ in the first equation arises from $\prod_{s,t\in\FF_q}\omega_{s,t}^{-1/2}\sim q^{q^2}$, $n^{q^2-1}$ arises from the scaling when approximating the summation by integration, and the second equation follows by the scaling and translation: $q\sqrt{n}(z_{s,t}-q^{-2})\mapsto z_{s,t}$.
Thinking of $\vec z=(z_{\sigma,\tau})_{(\sigma,\tau)\neq(0,0)}$ as a $(q^2-1)$-dimensional vector and letting 
	$\cM=(\cM_{(\sigma,\tau),(\sigma',\tau')})_{(\sigma,\tau),(\sigma',\tau')\neq(0,0)}$ be the $(q^2-1)\times(q^2-1)$-matrix with entries
	$\cM_{(\sigma,\tau),(\sigma',\tau')}=1+\vecone\{(\sigma,\tau)=(\sigma',\tau')\}$, we see that
	$\sum_{s,t\in\FF_q}z_{st}^2=\scal{\cM\vec z}{\vec z}.$
The eigenvalues of $\cM$ are $1$ and $q^2$, the latter with multiplicity one.
Hence, (\ref{eqLaplace2}) yields
	\begin{align}\label{eqLaplace3}
	S&=
		(1+o(1))qn^{\frac{q^2-1}2}\int_{\mathbb{R}^{q^2-1}}\exp\left\{-\frac{1}{2}\scal{\cM\vec z}{\vec z}\right\}\dd\vec z
		=(1+o(1))\frac{q}{\sqrt{\det\cM}}\bc{2\pi n}^{\frac{q^2-1}2}\sim\bc{2\pi n}^{\frac{q^2-1}2}.
	\end{align}
Finally, since $H(\bar{\omega})=2\ln q$ and $\Phi(\bar{\omega})=-2\ln q$, the assertion follows from (\ref{eqLaplace1}) and (\ref{eqLaplace3}).
\end{proof}

\begin{proof}[Proof of \Prop~\ref{Prop_smm}]
	Combining \Lem s~\ref{Lemma_EZ} and~\ref{Lemma_EZ2}, we see that for every small enough $\eps>0$,
	\begin{align*}
	\lim_{n\to\infty}\pr\brk{\Var[\cZ_\eps(\vA,\vy)|\vm]\leq 5\eps^{1/4} q^{n-\vm}}&=1.
	\end{align*}
	Hence, Chebyshev's inequality yields
	\begin{align*}
	\lim_{n\to\infty}\pr\brk{Z(\vA,\vy)=0}&\leq
	\lim_{n\to\infty}\pr\brk{Z(\vA,\vy)<(1-\eps^{1/16})q^{n-\vm}}\leq
	\lim_{n\to\infty}\Erw\brk{\pr[\cZ_{\eps}(\vA,\vy)<(1-\eps^{1/16})q^{n-\vm}\mid\vm]}
	< 5\eps^{1/8}.
	\end{align*}
	Taking $\eps\to0$ shows $\lim_{n\to\infty}\pr\brk{\vy\in\Im(\vA)}=1$, whence $\lim_{n\to\infty}\pr\brk{\rk(\vA)=\vm}=1$.
\end{proof}

\bigskip\noindent{\bf Acknowledgment.}
We thank Charilaos Efthymiou, Alan Frieze, Mike Molloy, Wesley Pegden and Claudius Zibrowius for helpful discussions.

\begin{appendix}

\section{Proof of \Thm~\ref{Cor_freezing}}\label{Sec_freezing}

\noindent
The proof is based on combining \Thm~\ref{Thm_SAT} with the argument from~\cite{AchlioptasMolloy}, where
\Thm~\ref{Cor_freezing} was established in the special case $q=2$.
As the solutions of $\vA x = \vy$ are translates to the homogeneous system, it suffices to prove \Thm~\ref{Cor_freezing} given that $\vy=0$. We will study the geometry of the set of solutions by way of the $2$-core of the random hypergraph $G(\vA)$. Its importance is due to the following observation. 
 If a variable $x_i$ is contained in exactly one equation of the linear system, then given the values of the other variables in that equation we can always set $x_i$ such that the equation is satisfied.
In other words, the value of $x_i$ can be expressed as a linear combination of other variables in the equation.
Moreover, if $x_i$ is not contained in any equation, then it can be assigned any value.
In either case we can remove $x_i$ and the equation, if there is one, containing it, and
any solution of the reduced system can be extended to a set of solutions to the original system.
Repeating this reduction leaves us with a system $\vA_*x_*=0$ where every variable is contained in at least two equations.
The set of remaining variables {is precisely} the vertex set of the $2$-core of the hypergraph $G(\vA)$, and the remaining equations correspond precisely to the hyperedges of the $2$-core.

The process of reducing the system $\vA x=0$ to its $2$-core is called the {\em peeling process}.
Since the intermediate stages of the peeling process will be important to us, we need to be a bit more formal.
Suppose that $\pi$ is a permutation of the set $[n]$.
Then the peeling process proceeds as follows.
Starting from the entire linear system $\vA x=0$, at each time we check if there is a variable $x_i$ that occurs in no more than one equation.
If so, we choose variable with the index $i$ that has the smallest value $\pi(i)$ and remove this variable $x_i$ along with the equation in which it occurs (if any).

The peeling process gives rise to a directed graph $D(\vA,\pi)$, constructed inductively as follows.
Initially, $D(\vA,\pi)$ is an empty graph with vertex set $\{x_1,\ldots,x_n\}$.
When the peeling process deletes variable $x_i$ we add to $D(\vA,\pi)$ a directed edge $(x_j,x_i)$ for every variable $x_j$ that occurs in the equation containing $x_i$ at the time of its removal (if there is such an equation).
Further, for each variable $x_i$ we let $R(\vA,\pi,x_i)$ be the set of all variables $x_j$ that can be reached from $x_i$ via a directed path in $D(\vA,\pi)$.
Further, following~\cite{AchlioptasMolloy}, we define
a {\em flippable cycle} of $\vA$ as a set of variables $X=\{x_{i_1},\ldots,x_{i_t}\}$ such that the set of hyperedges incident with $X$ in $G(\vA)$ can be ordered as $e_1,\ldots, e_t$ such that each $x_{i_h}$ lies in $e_h$ and $e_{h+1}$ (addition mod t) but not in any other hyperedges of $G(\vA)$. A {\em core flippable cycle} is a flippable cycle whose vertices all belong to the $2$-core of $G(\vA)$.
The following lemma paraphrases~\cite[\Thm~5, \Lem~45 and \Prop~48]{AchlioptasMolloy}.

\begin{lemma}\label{Thm_DAMSOM}
Suppose that $k\geq3$ and $0<d\neq d_k^\star$.
Then \whp\ for the random linear system $\vA x=0$ there is a permutation $\pi$ of $[n]$ such that the following is true.
\begin{enumerate}[(i)]
\item There are no more than $\sqrt{\ln\ln n}$ variables that belong to core flippable cycles.
\item Any two solutions $x_*,x_*'$ of the reduced system $\vA_*x=0$ either agree on all variables that do not belong to core flippable cycles or disagree on $\Omega(n)$ variables.
\item {For all} $i\in[n]$ we have $|R(\vA,\pi,x_i)|\leq(\ln\ln n)\ln n$.
\item {For all} core flippable cycles $C$ we have
		$\sum_{x_i\in C}|R(\vA,\pi,x_i)|\leq\sqrt{\ln\ln n}\ln n.$
\end{enumerate}
\end{lemma}
Let $\cF(\vA,\pi)$ be the set of all variables $x_i$ that have in-degree $0$ in {$D(\vA, \pi)$.}

\begin{lemma}\label{Fact_extend}
\begin{enumerate}
\item Any solution to the reduced system $\vA_*x_*=0$ extends to a solution to the entire system $\vA x=0$.
\item If $x_*,x_*'$ are solutions to the reduced system that differ only on a variable set $U\subset\cF(\vA,\pi)$,  then they extend to solutions $x,x'$ that differ only on the set $\bigcup_{x_i\in U}R(\vA,\pi,x_i)$.
\item For any two solutions $x,x'$ that coincide on the variables in the $2$-core there exists a sequence $x^{(0)},\ldots,x^{(\ell)}$ of solutions with $x^{(0)}=x$, $x^{(\ell)}=x'$ such that for every $l<\ell$ the Hamming distance of $x^{(l)}$ and $x^{(l+1)}$ is upper-bounded by $\max_{i\in[n]}|R(\vA,\pi,x_i)|$.
\end{enumerate}
\end{lemma}
\begin{proof}
The extension for (1) is obtained by simply assigning the variables outside the $2$-core in the inverse order of their removal by the peeling process.
For (2) we run this extension process from $x$ and $x'$ in parallel.
Then we can assign all the non-core variables that do not belong to the set $\bigcup_{x_i\in U}R(\vA,\pi,x_i)$ of variables impacted by $U$ the same.
A similar argument yields (3).
Indeed, reorder the variables such that $x_1,\ldots,x_r$ are the variables outside the $2$-core, removed by the peeling process in this order.
Let $i\leq r$ be the largest index such that $x_i\neq x_i'$.
Altering only the values that $x'$ assigns to the variables in $\in R(\vA,\pi,x_i)$, we can obtain a solution $x''$ such that $x_h''=x_h$ for all $h\geq i$ and such that $x_h''=x_h'$ for all $h\in[i-1]\setminus R(\vA,\pi,x_i)$.
In particular, $x',x''$ have Hamming distance at most $|R(\vA,\pi,x_i)|$ and the largest variable index on which $x,x''$ differ is less than the largest index where $x,x'$ differ.
Proceeding inductively, we obtain the desired sequence.
\end{proof}

\begin{proof}[Proof of \Thm~\ref{Cor_freezing}]
If $d<d_k^\star$, then by \Thm~\ref{Thm_Mike}  the $2$-core of $G(\vA)$ is empty \whp\
Therefore, \Lem~\ref{Fact_extend} (3) and \Lem~\ref{Thm_DAMSOM} show that there is just one cluster \whp. This proves part (i).

Now assume that $d_k^\star<d<d_k$.
Consider a solution $x_*$ of the reduced system $\vA_*x=0$.
Then \Lem~\ref{Fact_extend} (1) shows that $x_*$ extends to a solution $x$ of the entire system.
For two solutions $x_1,x_2$ of $\vA_*x=0$, we write $x_1\approx x_2$ if they agree on all variables in the 2-core but possibly some variables belonging to flippable cycles. The relation $\approx$ is reflexive, symmetric and transitive, and thus provides a partition $\cC$ of the solutions of $\vA_*x=0$. Let $C\in\cC$  be a part in the partition. Let $S(C)$ denote the set of solutions to $\vA x=0$ extended from any $x\in C$. By \Lem~\ref{Thm_DAMSOM} and \Lem~\ref{Fact_extend}, $\Sigma(\vA,0)=\{S(C):\ C\in\cC\}$ forms the set of solution clusters of $\vA x=0$.

Finally, 
since the matrix $\vA$ has full rank \whp\ by \Thm~\ref{Thm_SAT}, so does the reduced matrix $\vA_*$.
Therefore, \Thm~\ref{Thm_Mike} shows that \whp\ $\vA_*$ has nullity
	\begin{equation}
	{n_*(\vA)-m_*(\vA)}=n(\rho_{k,d}-d\rho_{k,d}^{k-1}+d(1-1/k)\rho_{k,d}^k+o(1)).
	\end{equation}
Furthermore, by  \Lem~\ref{Thm_DAMSOM}, $\log_q |C|=O(\sqrt{\ln\ln n} \ln n)$ \whp, and thus 
\begin{equation}\label{eq_Cor_freezing_1}
\log_q |\cC| = n(\rho_{k,d}-d\rho_{k,d}^{k-1}+d(1-1/k)\rho_{k,d}^k+o(1)) - O(\sqrt{\ln\ln n} \ln n).
\end{equation} 
Hence, (\ref{eqCor_freezing}) follows from (\ref{eq_Cor_freezing_1}) and the fact that
$\Sigma(\vA,\vy)=\Sigma(\vA,0)=|\cC|$.

\end{proof}

\end{appendix}

\end{document}